\documentclass[10pt,a4paper]{article}

\usepackage[utf8]{inputenc} %
\usepackage[T1]{fontenc} %

\usepackage{hyperref}
\usepackage{amsmath,amssymb,amsthm}
\usepackage{graphics}
\usepackage{graphicx}
\usepackage{epsfig}
\usepackage{verbatim}
\usepackage{times}
\usepackage{caption}
\usepackage{subcaption}
\usepackage{pifont}
\usepackage{shadow}
\usepackage{listings}
\usepackage{latexsym}
\usepackage{pdfsync}
\usepackage{bbm}
\usepackage{amsmath, amsfonts}
\usepackage{dsfont}  
\usepackage{flafter}  
\usepackage{placeins}
\usepackage{mathrsfs}
\usepackage{xcolor}
\usepackage{paralist}
\usepackage[normalem]{ulem}
\usepackage{xspace}
\usepackage{nicefrac}
\usepackage{epstopdf}
\usepackage{tikz}
\usepackage{pgfplots,pgfplotstable}
\usepackage{ifthen}
\usepackage{bm}
\usepackage{setspace}

\newcommand{\zerobf}{\boldsymbol{0}}

\newtheorem{hypo}{Assumption}[section]
\newtheorem{defi}[hypo]{Definition}
\newtheorem{theo}[hypo]{Theorem}
\newtheorem{lema}[hypo]{Lemma}

\newtheorem{prop}[hypo]{Proposition}
\newtheorem{rem}[hypo]{Remark}

\newcommand{\ba}{{\mathbf a}}

\newcommand{\be}{{\mathbf e}}
\newcommand{\bff}{{\mathbf f}}

\newcommand{\bn}{{\mathbf n}}

\newcommand{\bv}{{\mathbf v}}
\newcommand{\bw}{{\mathbf w}}
\newcommand{\bx}{{\mathbf x}}
\newcommand{\by}{{\mathbf y}}
\newcommand{\bz}{{\mathbf z}}

\newcommand{\bPsi}{\bm{\Psi}}

\newcommand{\bV}{{\mathbf V}}
\newcommand{\bW}{{\mathbf W}}

\newcommand{\fm}{{\mathfrak m}}
\newcommand{\fj}{{\mathfrak j}}

\newcommand{\fp}{{\mathfrak p}}
\newcommand{\fq}{{\mathfrak q}}
\newcommand{\fv}{{\mathfrak v}}

\newcommand{\bfv}{\bm{\mathfrak v}}
\newcommand{\bfw}{\bm{\mathfrak w}}

\newcommand{\IQ}{{\mathbb Q}}
\newcommand{\IR}{{\mathbb R}}

\newcommand{\IN}{{\mathbb N}}
\newcommand{\IZ}{{\mathbb Z}}

\newcommand\xscal[3]{\left< #1, #2 \right>_{#3}}

\newcommand\xnorm[2]{\left\lVert #1 \right\rVert_{#2}}

\newcommand\mA{\mathcal{A}}

\newcommand\mB{\mathcal{B}}

\newcommand\mG{\mathcal{G}}

\newcommand\mO{\mathcal{O}}

\newcommand\bmV{\bm{\mathcal V}}
\newcommand\inc{\textsf{inc}}

\newcommand\loc{\textsf{loc}}

\newcommand\apriori{\emph{a priori} }
\newcommand\aposteriori{\emph{a posteriori} }

\DeclareMathOperator{\Laplace}{\Delta}
\DeclareMathOperator{\Div}{div}
\DeclareMathOperator{\Curl}{\mathbf{curl}}

\newcommand{\Ltwo}       {\mathrm{L}^2}
\newcommand{\Ltwoloc}    {\mathrm{L}_{\loc}^2}

\newcommand{\Hhalf}      {\mathrm{H}^{\nicefrac12}}
\newcommand{\Hone}       {\mathrm{H}^1}
\newcommand{\Honezero}       {\mathrm{H}_0^1}

\newcommand{\Hs}         {\mathrm{H}}

\renewcommand\imath{\mathrm{i}}

\newcommand\epsi{\varepsilon}

\renewcommand\Re{\mathrm{Re}}
\renewcommand\Im{\mathrm{Im}}

\newcommand{\eg}{\emph{e.\,g\mbox{.}},\xspace}
\newcommand{\ie}{\emph{i.\,e\mbox{.}},\xspace}
\newcommand{\etal}{\emph{et\,al\mbox{.}}\xspace}
\newcommand{\cf}{\textrm{cf.}\xspace}

\pgfplotsset{compat=1.14} 
\makeatletter

\@addtoreset{equation}{section}
\makeatother
{\tiny }
\begin{document}
\begin{center} 

{\Large{Surface homogenization of an array of Helmholtz resonators for a viscoacoustic model using two-scale convergence}}\\[0.5cm]

{\large{Kersten Schmidt$^{a}$, Adrien Semin$^{a}$}}\\[0.5cm]

{\small $a$: Technische Universität Darmstadt, Fachbereich Mathematik,
  AG Numerik und Wissenschaftliches Rechnen,
  Dolivostrasse 15, 64293 Darmstadt,Germany}\\
\end{center}

{\small \noindent \textbf{Corresponding author:} Kersten Schmidt,
  Technische Universität Darmstadt, Fachbereich Mathematik,
  AG Numerik und Wissenschaftliches Rechnen,
  Dolivostrasse 15, 64293 Darmstadt,Germany\\
  E-mail: kschmidt@mathematik.tu-darmstadt.de}\\[0.5cm]

\begin{abstract}
  We derive the weak limit of a linear viscoacoustic model in an acoustic liner
that is a chamber connected to a periodic repetition of elongated chambers --
the Helmholtz resonators. %
As model we consider the time-harmonic and linearized compressible
Navier-Stokes equations for the acoustic velocity and pressure. %
Following the approach in Schmidt~\etal, J. Math. Ind 8:15, 2018 for the viscoacoustic transmission problem of multiperforated plates
the viscosity is scaled as $\delta^{4}$ with the period $\delta$ of the array
of chambers and the size of the necks as well as the wall thickness like
$\delta^2$
such that the viscous boundary layers
are of the order of the size of the necks. %
Applying the method of two-scale convergence 
we obtain with a stability assumption in the limit 
$\delta \to 0$ that the acoustic pressure fulfills the Helmholz equation with
impedance boundary conditions. %
These boundary conditions depend on the frequency, the length of the
resonators and through the effective Rayleigh conductivity -- that can be
computed numerically -- on the shape of their necks. We compare the limit
model to semi-analytical models in the literature.
\end{abstract}





\tableofcontents

\section{Introduction}

The noise emission from aircraft gas turbines, car engines and several other
industrial applications is a matter of high concern. Its reduction is of major
public interest since it affects health and life of the community. This noise
reduction is also of major industrial interest. Especially, nowadays combustion
processes create acoustic sources of higher intensity in aircraft engines,
which in their turn create acoustic instabilities around particular frequencies
and may even harm the live time of the gas turbine. Engineers study liners,
which are perforated wall segments, that are able to suppress thermo-acoustic
instabilities and can provide a substantial amount of acoustic damping. An
important type of acounstic liner for aero-engine inlet and exhaust ducts
constitues of a array of small cells called Helmholtz resonators. Each of the
Helmholtz resonators -- the name goes back to H.~Helmholtz\cite{Helmholtz:1863}
consists of a rigid chamber filled with air that is connected to the
surrounding by a hole of a perforated plate, that is called orifice or even
neck. When excited with a fluctuating external pressure, that comes \eg from
the combusion process, the mass of the air inside and around the orifice moves
agains the large volume of compressible air inside the cavity, while viscous
effects cause dissipation of energy. This can be modeled as a
mass-spring-damping system.  The damping of this system is relatively small
except for frequencies close to the resonance frequencies of the liner where it
becomes considerably large.  The resonance frequencies and damping properties
depend mainly on the geometrical parameters of the resonators.

For a small Helmholtz resonator the first resonance frequency -- the Helmholtz resonance -- can be approximated by a simple formula\cite{Helmholtz:1863}
of Helmholtz
that has been justified by a mathematical analysis of the spectrum of the Laplace operator by Schweizer\cite{Schweizer:2015}
as well as by an asymptotic analysis of the Green's function of the Helmholtz equation\cite{Ammari.Zhang:2015}.
It is based on the observation that the pressure is almost uniform inside the resonator and
therefore the formula does not depend on the shape of the resonator, especially, if it is elongated or of compact size.
The simple formula has been improved by a so-called end-correction\cite{Morse:1948,Ingard:1953} of the aperture thickness. %
For elongated resonance chambers approximations of each resonance frequency can be obtained as
solutions of a nonlinear equations using semi-analytical formulas for the behaviour of the pressure
around the hole and in the chamber\cite{Panton.Miller:1975}.

For an effective damping a large number of Helmholtz resonators are arrayed. Due to the high number of resonators and the involved smaller
geometrical scales a direct numerical computations, \eg with the finite element method, would be not feasible. One is therefore
interested in equivalent problems in the domain above the resonators and multiperforated plate. %
To predict the frequency dependent damping properties of array of Helmholtz resonators impedance boundary conditions
has been proposed that depend on one complex function of the frequency -- the (normalized specified) acoustic impedance.
First, such a semi-analytic formula for the impedance has been introduced by Guess\cite{Guess:1975}
that depends again on the end-correction of the aperture thickness, to which nonlinear terms for high sound amplitudes 
can be added\cite{Ingard.Ising:1967,Singh.Rienstra:2014} that are especially important close to resonance
as well as terms in presence of a gracing flow\cite{Kooi.Sarin:1981,Rienstra.Singh:2018,SchulzDiss:2018}.
The impedance of an array of Helmholtz resonators is also computed numerically by coupling the 
instationary viscous Navier-Stokes equations in frequency 
domain in some region around the orifice with the Euler equation away from it by Lidoine~\etal\cite{Lidoine.Terasse.Abboud.Bennani:2007}.
A similar approach in time-domain on meshes refined close to the orifices is used in\cite{Tam.Kurbatskii:2000}  where only close to the orifices
viscosity is considered. %
However, it is not clear what is a good choice of the ``viscous'' region. %

In this contribution, we present an asymptotic homogenization of an
array of Helmholtz resonators of depth $L$, of small period $\delta$ and of even smaller diameter of the orifices
that is of the order $\delta^2$ taking the viscosity scaled like $\delta^4$ into account.
In this way the respective dominant behaviour in three different geometric scales is considered.
We derive impedance boundary conditions applying the method of two-scale convergence\cite{Nguetseng:1989,Allaire:1992} 
to the three different scales of the problem. To justify the weak convergence to the limit 
the stability estimate of the $\delta$-dependent problem has to assumed.
The impedance boundary conditions is expressed in terms of the 
\emph{effective Rayleigh conductivity} of a perforated plate~\cite{Schmidt.Semin.ThoensZueva.Bake:2018} (see 
\cite{Bendali.Fares.Piot.Tordeux:2012,Bendali.Fares.Laurens.Tordeux:2012} for zero viscosity)
and in terms of the reactance of the chambers that depends above all on their depths.
The effective Rayleigh conductivity is the Rayleigh conductivity~\cite{Rayleigh:1870,Rayleigh:1945} of one hole,
which describes the ratio of the fluctuating volume flow through the hole to the driving pressure difference
across the hole, divided by the area of one periodicity cell of the array. %
The effective Rayleigh conductivity depends on the geometrical
parameters, especially, size and shape of the necks of the Helmholtz resonators
and the distance between two resonators, as well as the physical parameters,
especially the acoustic viscosities and the excitation frequency.

Asymptotic homogenization for periodic transmission problems were performed 
for the Stokes equation with three scales\cite{Sanchez.Sanchez:1982},
with two scales for the Helmholtz equation\cite{Bonnet.Drissi.Gmati:2005},
using the periodic unfolding method\cite{Lukes.Rohan:2007} 
and the method of matched asymptotic expansion\cite{Delourme.Haddar.Joly:2012,Claeys.Delourme:2013}, 
also with impedance boundary conditions in the holes\cite{Semin.Schmidt:2018}.
Asymptotic homogenization for locally periodic transmission problems, where microstructures has finite size, and that takes 
the singular behaviour at the end of the microstructure into account,
was derived and justified for the Laplace equation\cite{Nazarov:2008,Delourme.Schmidt.Semin:2016}
and the Helmholtz equation\cite{Semin.Delourme.Schmidt:2018}.

The article is subdivided as follows. In Sec.~\ref{sec:descr-probl-main} we
define the model problem of the viscous acoustic equations in terms of the acoustic velocity and
pressure and the equivalent impedance boundary conditions on the array of
resonators for the velocity and pressure. We also give
as main result the weak
convergence of the velocity and the pressure to their limits that fullfill
a Helmholtz problem with the derived equivalent impedance boundary
condition. For this result the assumption of an \emph{a priori} stability result is needed
that shall be proved in a forthcoming article.
Sec.~\ref{sec:deriv-just-main} is dedicated to the
proof of the weak convergence to a limit using two-scale convergence
step by step in different asymptotic regions where finally the limit model including impedance boundary conditions is obtained.
Finally, in Sec.~\ref{sec:numerical-results}, the equivalent impedance boundary
conditions are studied numerically and compared with the established model of Guess\cite{Guess:1975},
both locally based on the reactance and resistance curves and resonance frequencies 
as well as macroscopically based on the dissipation behaviour of an array of Helmholtz resonators in a duct.


\section{Description of the problem and main results}
\label{sec:descr-probl-main}

\subsection{Description of the problem}

We consider a three-dimensional domain $\Omega$ that is open, simply
connected and bounded with smooth boundary $\partial\Omega$. %
For the sake of simplicity we consider that $\Omega$ is 
included in the half-space $\IR^2 \times \IR_+$ 
such that its boundary $\partial \Omega$ gives a non-empty intersection with the
plane $\lbrace x_3 = 0 \rbrace$.

We consider the surface $\Gamma$ as a parallelepipedic subset of $\partial \Omega
\cap \lbrace \bx \in \IR^3, x_3 = 0\rbrace$. We extend then the domain $\Omega$ to a
domain containing an array of Helmholtz resonators. We assume this
array to be periodic, \ie there exists two fixed vectors $\ba_1$ and
$\ba_2$ such that the centered parallelogram $\mA$ spanned by the
vectors $\ba_1$ and $\ba_2$ is of area equal to $1$, and there exists
$\delta > 0$ such that the set centers of apertures of resonators is
given by (see Fig.~\ref{fig:example_canonical_resonator}a)
\begin{equation}
\label{eq:set_centers}
\Gamma^\delta := \Gamma \cap \big( \delta \ba_1 \IZ + \delta \ba_2
\IZ \big).
\end{equation}
For simplicity, we assume that there exists $L_1, L_2 > 0$ with $L_1 / L_2 \in
\IQ$ such that
\begin{equation*}
\Gamma = \lbrace \bx = s_1 \ba_1 + s_2 \ba_2 \quad \text{with} \quad (s_1,s_2) \in
(0,L_1) \times (0,L_2) \rbrace,
\end{equation*}
and $\delta$ is chosen such that $L_1 / \delta$ and $L_2 / \delta$ are positive
integers, \ie the number of Helmholtz resonators is equal to $L_1L_2/\delta^2$. 
To define the resonator chamber we introduce its cross
section $\mA_C \subset \mA$ that is a two-dimensional smooth open and simply-connected domain, and two constants
$d_0,h_0 > 0$. For $(n_1,n_2) \in \IZ^2$ such that
$\bx^\delta_\Gamma := \delta n_1 \ba_1 + \delta n_2 \ba_2 \in \Gamma^\delta$,
the resonator $\Omega^\delta_H(\bx^\delta_\Gamma)$ consists of a chamber part
\begin{equation}
\label{eq:resonator_chamber}
\Omega^\delta_C(\bx^\delta_\Gamma) := \bx^\delta_\Gamma + \delta
\mA_C \times (-L, -\delta^2 h_0),
\end{equation}
and a neck part
\begin{equation}
\label{eq:resonator_neck}
\Omega^\delta_N(\bx^\delta_\Gamma) := \bx^\delta_\Gamma +
\delta^2 \Omega_N
%
\end{equation}
with the bounded, open and simply connected Lipschitz domain $\Omega_N \subset \IR^2 \times (-h_0,0)$
where $0 \in \overline{\Omega_N}$ and the submanifolds $\partial\Omega_N \cap \IR^2\times\{0\}$ and $\partial\Omega_N \cap \IR^2\times\{-h_0\}$
are non-empty and smooth. %
Moreover, we consider $\delta$ such that $\delta\left(\partial\Omega_N \cap \IR^2\times\{-h_0\}\right) \subset \mA_C \times \{-\delta h_0\}$.
A chamber and neck builds a resonator $\Omega^\delta_H(\bx^\delta_\Gamma) =
\Omega^\delta_C(\bx^\delta_\Gamma) \cup
\Omega^\delta_N(\bx^\delta_\Gamma)$, and extending the domain $\Omega$ by the union of all resonsators 
we obtain the computational domain $\Omega^\delta$ whose closure
is defined by
\begin{equation}
\label{eq:Omega_delta}
\overline{\Omega^\delta} := \overline{\Omega} \cup \bigcup_{\bx^\delta_\Gamma \in
    \Gamma^\delta} \overline{\Omega^\delta_H(\bx^\delta_\Gamma)}.
\end{equation}

On the domain $\Omega^\delta$ we introduce the acoustic equations in
the framework of Landau and Lifschitz\cite{Landau.Lifschitz:1959} as
a perturbation of the Navier-Stokes equations around a stagnant
uniform fluid with mean density $\rho_0$.

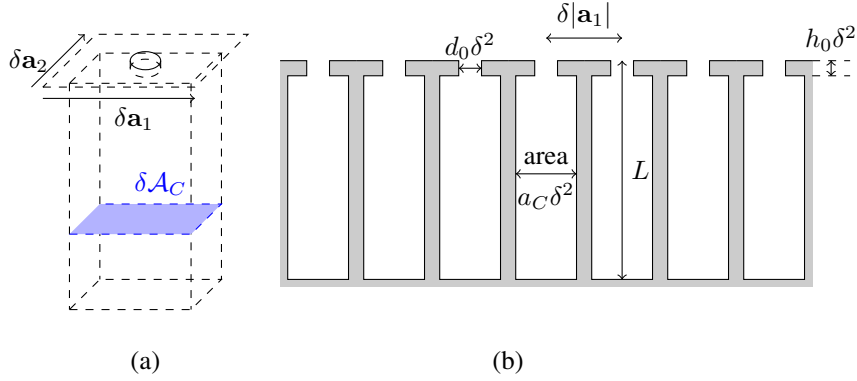
\begin{figure}[bt]
    \centering
    \null \hfill
    \begin{tikzpicture}
    \draw[dashed] (-1.35,-0.35) -- (0.65,-0.35) -- (1.35,0.35) -- (-0.65,0.35) --
    (-1.35,-0.35);
    
    \draw[->] (-1.35,-0.5) -- (0.65,-0.5) node [pos=0.6,below] {$\delta\ba_1$};
    \draw[->] (-1.5,-0.35) -- (-0.8,0.35) node [pos=0.5,left] {$\delta\ba_2$};
    \draw (0,0) ellipse (0.2 and 0.12);
    \draw[dashed] (0,-0.1) ellipse (0.2 and 0.12);
    \draw[dashed] (-0.2,0) -- (-0.2,-0.1);
    \draw[dashed] (0.2,0) -- (0.2,-0.1);
    \draw[dashed] (-1.0,-0.3) -- (0.6,-0.3) -- (1.0,0.1) -- (-0.6,0.1) --
    (-1.0,-0.3);
    \draw[dashed] (-1.0,-0.3) -- (-1.0,-3.3);
    \draw[dashed] (0.6,-0.3) -- (0.6,-3.3);
    \draw[dashed] (1.0,0.1) -- (1.0,-2.9);
    \draw[dashed] (-0.6,0.1) -- (-0.6,-2.9);
    \draw[dashed] (-1.0,-3.3) -- (0.6,-3.3) -- (1.0,-2.9) -- (-0.6,-2.9) --
    (-1.0,-3.3);
    \filldraw[fill=blue!30!white,draw=blue,dashed] (-1.0,-2.3) --
    (0.6,-2.3) -- (1.0,-1.9) -- (-0.6,-1.9) node[pos=0.5,above]
    {\textcolor{blue}{$\delta \mA_C$}};
    \node at (0,-4) {(a)};
    \end{tikzpicture}
    \hfill
    \begin{tikzpicture}
    \foreach \i in {-3,...,3}
    {
        \begin{scope}[xshift=\i cm]
        \fill[black!20!white] (0,-3) rectangle (1,0);
        \fill[white] (0.1,-2.9) rectangle (0.9,-0.2);
        \fill[white] (0.35,-0.3) rectangle (0.65,0.05);
        \draw (0,0) -- (0.35,0)-- (0.35,-0.2) -- (0.1,-0.2) -- (0.1,-2.9) -- (0.9,-2.9) -- (0.9,-0.2) -- (0.65,-0.2) -- (0.65,0) -- (1,0); 
        \end{scope}
    }
    \draw[dashed] (4,0) -- (4.5,0);
    \draw[dashed] (4,-0.2) -- (4.5,-0.2);
    \draw[<->] (0.5,0.3) -- (1.5,0.3) node[pos=0.5,above] {$\delta |\ba_1|$};
    \begin{scope}[xshift=-1cm]
    \draw[<->] (0.35,-0.1) -- (0.65,-0.1) node[pos=0.5,above] {$d_0\delta^2$};
    \end{scope}
    \draw[<->] (0.1,-1.5) -- (0.9,-1.5) node[pos=0.5,above]
    {area} node[pos=0.5,below] {$a_C \delta^2$};
    \begin{scope}[xshift=3.75cm]
    \draw[<->] (0.5,0) -- (0.5,-0.2) node[pos=0,above] {$h_0\delta^2$};
    \end{scope}
    \draw[<->] (1.5,-2.9) -- (1.5,0) node[pos=0.5,right] {$L$};
    \node at (0,-4) {(b)};
    \end{tikzpicture}
    \hfill\null
    \caption{(a) Example of one resonator (square-shaped constant
        cross-sections) that connects through $N_0=1$ hole. (b)
        Representation of the array of resonators (cut along one
        one periodicity direction).}
    \label{fig:example_canonical_resonator}
\end{figure}

We consider time-harmonic velocity $\bv^\delta$ and acoustic pressure
$p^\delta$ (the time regime is $\exp(-\imath \omega t)$, $\omega > 0$), which
are solutions of the coupled system
\begin{subequations}
    \label{eq:Navier_Stokes}
    \begin{align}
    - \imath \omega \bv^\delta + \tfrac{1}{\rho_0} \nabla p^\delta - \nu(\delta)
    \Laplace \bv^\delta - \nu'(\delta) \nabla \Div \bv^\delta 
    &= \bff, \quad \text{in }\Omega^\delta, \label{eq:Navier_Stokes:M}
    \\ 
    - \imath \omega p^\delta + \rho_0 c^2 \Div \bv^\delta 
    &= 0, \quad \text{in }\Omega^\delta, \label{eq:Navier_Stokes:C} \\
    \bv^\delta
    &= \zerobf , \quad \text{on } \partial
    \Omega^\delta, \label{eq:Navier_Stokes:B}
    \end{align}
\end{subequations}
with the speed of sound $c$, and the kinematic and secondary viscosities
$\nu(\delta), \nu'(\delta) > 0$
that we scale with the characteristic size of the holes $\delta^2$ as
\begin{equation}
\label{eq:asymptotic_scale_viscosity}
\nu(\delta) = \nu_0 \delta^4 \quad \text{and} \quad \nu'(\delta) =
\nu'_0 \delta^4,
\end{equation}
where $\nu_0, \nu_0'$ are independent of $\delta$. %
In this way the thickness of the boundary layer of the acoustic velocity 
at the rigid
wall\cite{andreev2008pointwise,auregan2001influence,Berggren.Bernland.Noreland:2018,Iftimie.Sueur:2010,Schmidt.Thoens.Joly:2014}
-- that is of order $O(\sqrt{\nu(\delta)})$ -- is of the order 
of characteristic size of the holes $\delta^2$ (see Fig.~\ref{fig:example_canonical_resonator}(b)).
Moreover, the source term $\bff$ is independent of
$\delta$ and compactly supported in $\Omega$ away from its boundary. 
Similar equations have been studied
for a stagnant flow\cite{howe1998acoustics,Iftimie.Sueur:2010,Landau.Lifschitz:1959,Rienstra.Hirschberg:2018} and
for the case that a mean flow is
present\cite{auregan2001influence,howe1979influence,howe1998acoustics,munz2007linearized,rienstra2011boundary}.
Finally, we embed the domain $\Omega^\delta$ and the associate linear
Navier-Stokes problem~\eqref{eq:Navier_Stokes} in a family of problems
that are $\delta$-dependent. 


In the following, we define the effective Rayleigh conductivity of a single hole that will be used then to 
define impedance boundary conditions for the limit of $(\bv^\delta,p^\delta)$ for $\delta\to0$.

\subsection{Effective Rayleigh conductivity}

To relate the pressure jump on different sides of a single hole to pressure and
velocity profiles and eventually the flux through the hole we consider as
characteristic problem an instationary Stokes system in scaled
coordinates\cite{Schmidt.Semin.ThoensZueva.Bake:2018}: Seek
$(\bfv,\fp) \in (\Honezero(\widehat{\Omega}))^3 \times \Ltwoloc(\widehat{\Omega})$
solution of
\begin{equation}
\label{eq:canonical_problem}
\begin{aligned}
- \imath \omega \bfv + \tfrac{1}{\rho_0}\nabla
\fp - \nu_0 \Laplace
\bfv & = \zerobf, && \quad \text{in } \widehat{\Omega}, \\
\Div \bfv & = 0, && \quad \text{in } \widehat{\Omega}, \\
\bfv & = \zerobf , && \quad \text{on } \partial
\widehat{\Omega}, \\
\lim_{S \to \infty} \fp_{|\widehat{\Gamma}_\pm(S)} &= \pm \tfrac{1}{2}\ ,
\end{aligned}
\end{equation}
with the family of surfaces
\begin{align*}
\widehat{\Gamma}_+(S) &:= \{ \bz \in \IR^3 : |\bz| = S, z_3 > 0 \}, \\
\widehat{\Gamma}_-(S) &:= \{ \bz \in \IR^3 : |\bz-(0,0,-h_0)| = S, z_3 < -h_0 \}
\end{align*}
and the rescaled extended aperture domain 
$\widehat{\Omega} := \cup_{S \in \IR^+}\widehat{\Omega}(S)$
connecting two half-spaces 
(see Fig.~\ref{fig:NNF}(c) for illustration) where
\begin{align*}
\widehat{\Omega}(S) := \widehat{\Omega}_A &\cup \{ \bz \in \IR^3 : |\bz| < S, z_3 > 0 \} \\
&\cup \{ \bz \in \IR^3 : |\bz-(0,0,-h_0)| < S, z_3 < -h_0 \}.
\end{align*}
Due to the scaling all other holes are moved towards infinity and canonical
problem considers the dominant phenomena on the scale of one hole, that is
viscosity and incompressibility of the acoustic velocity that is with the
acoustic pressure non-stationary.

The well-posedness of the characteristic problem is stated in the following
\begin{prop}
    \label{prop:well_posedness_aperture_problem}
    There exists a unique solution
    $(\bfv,\fp) \in (\Honezero(\widehat{\Omega}))^3 \times
    \Ltwoloc(\widehat{\Omega})$ of~\eqref{eq:canonical_problem}.
\end{prop}
\begin{proof}
    First we lift the condition for the pressure at infinity with a
    cut-off function
    \begin{equation*}
    \begin{aligned}
    \Theta(\bz) &= 
    \begin{cases}
    \tfrac{1}{2}, & z_3 > 0,\\
    \tfrac{1}{2}+\frac{z_3}{h_0}, & 0 > z_3 > -h_0,\\
    -\tfrac{1}{2}, & \text{otherwise}
    \end{cases}
    \end{aligned}
    \end{equation*}
    We decompose then the pressure $\fp$ as $\fp = \Theta + \tilde{\fp}$ and seek
    $\tilde{\fp}$ in the classical space $\Ltwo(\widehat{\Omega})$.  The
    variational formulation associated to~(\ref{eq:canonical_problem}) is: find
    $(\bfv,\tilde{\fp}) \in (\Honezero(\widehat{\Omega}))^3 \times
    \Ltwo(\widehat{\Omega})$
    such that for any
    $(\bfw,\fq) \in (\Honezero(\widehat{\Omega}))^3 \times
    \Ltwo(\widehat{\Omega})$,
    \begin{equation}
    \label{eq:canonical_problem:var}
    \begin{aligned}
    a(\bfv,\bfw) + b(\tilde{\fp},\bfw) & = \ell(\bfw), \\
    b(\fq,\bfv) & = 0,
    \end{aligned}
    \end{equation}
    where the sesquilinear forms $a$ and $b$ are given by
    \begin{equation*}
    \begin{aligned}
    a(\bfv,\bfw) & := - \imath \omega \int_{\widehat{\Omega}} \bfv \cdot \overline{\bfw}\,\text{d}\bz +
    \nu_0 \int_{\widehat{\Omega}} \nabla \bfv : \nabla \overline{\bfw}\,\text{d}\bz , &
    b(\fq,\bfv) & := -\tfrac{1}{\rho_0} \int_{\widehat{\Omega}} \fq \Div \overline{\bfv}\,\text{d}\bz,
    \end{aligned}
    \end{equation*}
    and the antilinear form $\ell$ by
    \begin{equation*}
    \ell(\bfw) := \tfrac{1}{\rho_0} \int_{\widehat{\Omega}} \nabla
    \Theta \cdot \overline{\bfw}\,\text{d}\bz\ .
    \end{equation*}
    The formulation~\eqref{eq:canonical_problem:var} has a saddle-point
    structure.  The sesquilinear form $a$ is continuous and elliptic on
    $(\Honezero(\widehat{\Omega}))^3$. The sesquilinear form $b$ defines a
    surjective operator $B : (\Honezero(\Omega))^3 \to \Ltwo(\Omega)$ by
    \begin{equation*}
    \left<B\bfv, \fq \right>_{\widehat{\Omega}}
    = b(\fq,\bfv), \quad \forall (\bfv,\fq) \in (\Honezero(\Omega))^3\times \Ltwo(\Omega),
    \end{equation*}
    with closed range.
    
    Following the theory of saddle-point problems\cite{BrezziFortinBook} and in
    view of Theorem~1.1 of the works of Brezzi\cite{brezzi1974existence},
    problem~(\ref{eq:canonical_problem:var}) admits a unique solution
    $(\bfv,\tilde{\fp}) \in (\Honezero(\widehat{\Omega}))^3 \times
    \Ltwo(\widehat{\Omega})$,
    and, hence, (\ref{eq:canonical_problem}) has a unique solution.
\end{proof}

Following the
formulation of the Rayleigh conductivity
$K_R$\cite{Rayleigh:1870,Rayleigh:1945} which describes the ratio of
the fluctuating volume flow to the driving pressure difference, we
introduce the \emph{effective Rayleigh conductivity}
$k_R$\cite{Bendali.Fares.Piot.Tordeux:2012,Schmidt.Semin.ThoensZueva.Bake:2018} as
\begin{equation}
\label{eq:KR}
k_R := \lim_{S \to \infty} \frac{\imath \omega \rho_0}{2} \Big( \int_{\widehat{\Gamma}_+(S)}
\bfv \cdot \bn - \int_{\widehat{\Gamma}_-(S)}
\bfv \cdot \bn \Big).
\end{equation}

\begin{prop}
    \label{propostion:effective_behaviour_kR}
    The above defined effective Rayleigh conductivity $k_R$ is well defined, \ie the integral of the normal flux of $\bfv$ on
    $\widehat{\Gamma}_\pm(S)$ tends to a finite, non-zero quantity as $S$ tends to
    infinity. Moreover $\Re(k_R) > 0$ and $\Im(k_R) < 0$.
\end{prop}
\begin{proof}
    Taking the conjugate of~\eqref{eq:canonical_NNF:1} leads to
    \begin{equation*}
    \imath \omega \overline{\bfv} + \tfrac{1}{\rho_0} \nabla_\bz \overline{\fp}
    - \nu_0 \Delta_\bz \overline{\bfv} = \zerobf.
    \end{equation*}  
    Multiplying this equation with $\bfv$ and integrating over the domain $\widehat{\Omega}(S)$ leads to
    \begin{equation}
    \label{eq:proof_IPP_Omega_AS}
    \imath \omega \xnorm{\bfv}{\Ltwo(\widehat{\Omega}(S))}^2
    + \tfrac{1}{\rho_0} \xscal{\nabla_\bz \overline{\fp}}{\bfv}{\widehat{\Omega}(S)}
    - \nu_0 \xscal{\Delta_\bz \overline{\bfv}}{\overline{\bfv}}{\widehat{\Omega}(S)} = 0.
    \end{equation}
    Let us treat now the two scalar product terms using the Gauss' theorem. Using~\eqref{eq:canonical_NNF:2}
    leads to
    \begin{equation*}
    \xscal{\nabla_\bz \overline{\fp}}{\bfv}{\widehat{\Omega}(S)} = \xscal{\overline{\fp}}{\bfv \cdot \bn}{ \partial \widehat{\Omega}(S)}
    \end{equation*}
    Using~\eqref{eq:canonical_NNF_infinity} (which is also valid for the conjugate complex),
    the definition of the effective Rayleigh conductivity stated by~\eqref{eq:KR} and the spherical
    harmonic decompositions of $\fp$ and $\bfv$ leads to
    \begin{equation*}
    \lim_{S \to \infty} \xscal{\nabla_\bz \overline{\fp}}{\bfv}{\widehat{\Omega}(S)} = \frac{k_R}{\imath \omega \rho_0}
    \end{equation*}
    Similarly, the study of the $\Ltwo$-scalar product between $\Delta_\bz \overline{\bfv}$ and $\bfv$
    leads to
    \begin{equation*}
    - \lim_{S \to \infty}  \xscal{\Delta_\bz \overline{\bfv}}{\overline{\bfv}}{\widehat{\Omega}(S)}
    = \xnorm{\nabla \bfv}{\Ltwo(\widehat{\Omega})}^2
    \end{equation*}
    Taking then the limit in~\eqref{eq:proof_IPP_Omega_AS} as $S \to \infty$ leads to
    \begin{equation*}
    \imath \omega \xnorm{\bfv}{\Ltwo(\widehat{\Omega}(S))}^2
    + \nu_0 \xnorm{\nabla \bfv}{\Ltwo(\widehat{\Omega}(S))}^2 =
    - \frac{k_R}{\imath \omega \rho_0^2} = \frac{1}{\omega \rho_0^2} \big( - \Im(k_R) + \imath \Re(k_R) \big)
    \end{equation*}
    and therefore $\Re(k_R) > 0$ and $\Im(k_R) < 0$.
\end{proof}

\subsection{Weak convergence to a limit problem with impedance boundary conditions}

As $\delta$ tends to $0$, we expect that the solution
$(\bv^\delta,p^\delta)$ tends to a finite, non-trivial limit solution
$(\bv_0,p_0)$ in the half-space $\Omega$, and we expect this limit
term to be solution of an inviscid Helmholtz problem posed on
$\Omega$.

\begin{defi}[Limit problem]
    \label{defi:limit_problem}
    We define the limit problem $(\bv_0,p_0)$ as solution of
    \begin{equation}
    \label{eq:limit_term_complete}
    \begin{aligned}
    - \imath \omega \bv_0 + \tfrac{1}{\rho_0} \nabla p_0 & = \bff,
    &&
    \text{in } \Omega, \\
    - \imath \omega p_0 + \rho_0 c^2 \Div \bv_0 & = 0, && \text{in
    }\Omega, \\
    \Big( \tfrac{\imath c \rho_0}{a_C} \cos \big(\tfrac{\omega L}{c}\big) -
    \tfrac{\imath \omega \rho_0}{k_R} \sin \big(\tfrac{\omega L}{c}\big) \Big)
    \bv_0 \cdot \bn -  \sin \big(\tfrac{\omega L}{c} \big)
    p_0 & = 0, && \text{on }\Gamma, \\
    \bv_0 \cdot \bn & = 0, && \text{on }\partial \Omega \setminus
    \Gamma.
    \end{aligned}
    \end{equation}
\end{defi}

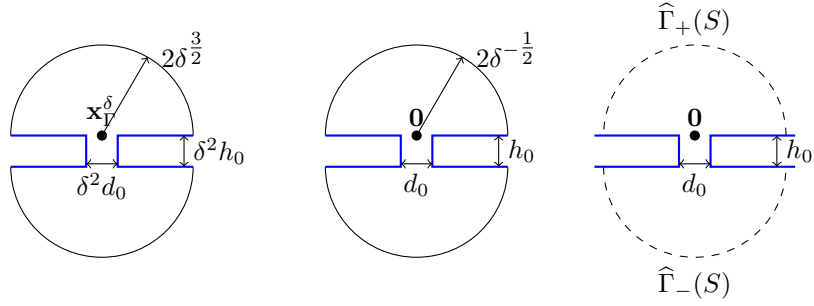
\begin{figure}[bt]
    \null\hfill
    \begin{tikzpicture}[baseline, scale=0.6]
    \draw[white] (-2.5,-3) rectangle (2.5,3);
    \draw (2,0.345) arc(0:180:2);
    \draw (-2,-0.345) arc(180:360:2);
    \draw[blue, thick] (-2,0.345) -- (-0.346,0.345) -- (-0.346,-0.345) -- (-2,-0.345);
    \draw[blue, thick] (2,0.345) -- (0.346,0.345) -- (0.346,-0.345) --
    (2,-0.345);
    \draw[->] (0,0.345) ++ (0,0) -- ++ (60:2) node [pos=1.1,right]
    {$2\delta^{\tfrac{3}{2}}$}; 
    \draw[<->] (1.8,-0.345) -- (1.8,0.345) node[pos=0.5, right]
    {$\delta^2 h_0$};
    \draw[<->] (-0.346,-0.28) -- (0.346,-0.28) node[pos=0.5, below]
    {$\delta^2 d_0$};
    \filldraw (0,0.345) circle (0.1) node [above] {$\bx^\delta_\Gamma$};
    \node at (0,-4) {\small (a) Extended neck domain $\Omega^\delta_A(\bx^\delta_\Gamma)$};
    \end{tikzpicture}
    \hfill
    \begin{tikzpicture}[baseline, scale=0.6]
    \draw[white] (-2.5,-3) rectangle (2.5,3);
    \draw (2,0.345) arc(0:180:2);
    \draw (-2,-0.345) arc(180:360:2);
    \draw[blue, thick] (-2,0.345) -- (-0.346,0.345) -- (-0.346,-0.345) -- (-2,-0.345);
    \draw[blue, thick] (2,0.345) -- (0.346,0.345) -- (0.346,-0.345) --
    (2,-0.345);
    \draw[->] (0,0.345) ++ (0,0) -- ++ (60:2) node [pos=1.1,right]
    {$2\delta^{-\tfrac{1}{2}}$}; 
    \draw[<->] (1.8,-0.345) -- (1.8,0.345) node[pos=0.5, right]
    {$h_0$};
    \draw[<->] (-0.346,-0.28) -- (0.346,-0.28) node[pos=0.5, below]
    {$d_0$};
    \filldraw (0,0.345) circle (0.1) node [above] {$\zerobf$};
    \node at (0,-4) {\small (b) Rescaled extended};
    \node at (0,-4.6) {\small neck domain $\widehat{\Omega}(2\delta^{-1/2})$};
    \end{tikzpicture}
    \hfill
    \begin{tikzpicture}[baseline, scale=0.6]
    \draw[white] (-2.5,-3) rectangle (2.5,3);
    \draw[dashed] (2,0.345) arc(0:180:2) node[pos=0.5, above] {$\widehat{\Gamma}_+(S)$};
    \draw[dashed] (-2,-0.345) arc(180:360:2) node[pos=0.5, below] {$\widehat{\Gamma}_-(S)$};
    \draw[blue, thick] (-2.2,0.345) -- (-0.346,0.345) -- (-0.346,-0.345) -- (-2.2,-0.345);
    \draw[blue, thick] (2.2,0.345) -- (0.346,0.345) -- (0.346,-0.345) --
    (2.2,-0.345);
    \draw[<->] (1.8,-0.345) -- (1.8,0.345) node[pos=0.5, right]
    {$h_0$};
    \draw[<->] (-0.346,-0.28) -- (0.346,-0.28) node[pos=0.5, below]
    {$d_0$};
    \filldraw (0,0.345) circle (0.1) node [above] {$\zerobf$};
    \node at (0,-4) {\small (c) Limit rescaled extended};
    \node at (0,-4.6) {\small neck domain $\widehat{\Omega}$};
    \end{tikzpicture}
    \hfill\null
    \caption{An extended neck domain in the original coordinates $\bx$ (left), in
        rescaled coordinates $\bz := \delta^{-2} (\bx-\bx^\delta_\Gamma)$ (middle)
        and the limit of the latter when $\delta \to 0$ (right).}
    \label{fig:NNF}
\end{figure}

Note that, problem~(\ref{eq:limit_term_complete}) is equivalent to a problem for
the limit pressure $p_0 \in \Hone(\Omega)$ only, that is given by
\begin{equation}
\label{eq:limit_term_complete_pressure}
\begin{aligned}
\Delta p_0 + \tfrac{\omega^2}{c^2} p_0 & = \rho_0 \Div \bff, &&
\text{in } \Omega, \\
\Big( \tfrac{c}{\omega a_C} \cos \big(\tfrac{\omega L}{c}\big) -
\tfrac{1}{k_R} \sin \big(\tfrac{\omega L}{c}\big) \Big) \nabla p_0 \cdot \bn
-  \sin \big(\tfrac{\omega L}{c}\big) 
p_0 & = 0, && \text{on }\Gamma, \\
\nabla p_0 \cdot \bn & = 0, && \text{on }\partial \Omega \setminus
\Gamma,
\end{aligned}
\end{equation}
where
$\bv_0 := \tfrac{\imath}{\omega} \big( \bff - \tfrac{1}{\rho_0} \nabla p_0
\big)$
follows. It is also equivalent to a problem for the limit velocity
$\bv_0 \in \Hs(\Div,\Omega) \cap \Hs(\Curl,\Omega)$ only that is
\begin{equation}
\label{eq:limit_term_complete_velocity}
\begin{aligned}
\nabla \Div \bv_0 + \tfrac{\omega^2}{c^2} \bv_0 & = \tfrac{\imath
    \omega}{c^2} \bff, && \text{in } \Omega, \\
\Curl \bv_0 & = - \tfrac{1}{\imath \omega} \Curl \bff, && \text{in }
\Omega,
\\
\Big( \tfrac{c}{ \omega a_C} \cos \big(\tfrac{\omega L}{c}\big) -
\tfrac{1}{k_R} \sin \big(\tfrac{\omega L}{c}\big) \Big) \bv_0 \cdot \bn
\\ -
\Big( \tfrac{c^2}{\omega^2} \sin
\big(\tfrac{\omega L}{c}\big) \Big) \Div \bv_0 & = 0, && \text{on }\Gamma, \\
\bv_0 \cdot \bn & = 0, && \text{on }\partial \Omega \setminus
\Gamma,
\end{aligned}
\end{equation}
where $p_0 := \tfrac{-\imath \rho_0 c^2}{\omega} \Div \bv_0$ follows.

The nature of the boundary condition on $\Gamma$ depends on the value of $\sin
\tfrac{\omega L}{c}$
\begin{enumerate}
    \item If $\sin \big( \tfrac{\omega L}{c} \big) = 0$, \ie if
    $\tfrac{\omega}{c}$ corresponds to a characteristic wavelength of the
    one-dimensional Helmholtz problem in a domain of size $L$, then the
    boundary condition on $\Gamma$ becomes $\bv_0 \cdot \be_3 = 0$, therefore
    $\bv_0 \cdot \bn = 0$ on the whole boundary $\partial \Omega$.
    \item \emph{A contrario}, if $\sin \big( \tfrac{\omega L}{c} \big) \not= 0$, the
    right-hand side does not vanish, and denoting here by $\bn := - \be_3$
    the unit outward normal vector on $\Gamma$, one gets the impedance
    condition
    \begin{equation}
    \label{eq:impedance_bc} \Big( - \tfrac{\imath c}{a_C(\bx_\Gamma)}
    \cot \big( \tfrac{\omega L}{c} \big) + \tfrac{\imath \omega}{k_R(\bx_\Gamma)}
    \Big) \bv_0(\bx_\Gamma) \cdot \bn + \tfrac{1}{\rho_0} p_0(\bx_\Gamma)
    = 0,
    \end{equation}
    and this equation gives an acoustic impedance $Z(\omega)$ of the same
    nature as the one derived by Rienstra and Singh\cite[Eq.~(14)]{Rienstra.Singh:2018}.
\end{enumerate}
We give the existence and uniqueness result of the limit problem:
\begin{lema}[Existence and uniqueness of the limit problem]
    \label{lema:well-posedness_limit}
    Let $\bff \in \Hs(\Div,\Omega)$. 
    Then, the limit problem~\eqref{eq:limit_term_complete} is well-posed, \ie
    admits a unique solution
    $(\bv_0,p_0) \in \Hs(\Div,\Omega) \times \Hone(\Omega)$, except for
    frequencies $\omega \in \Lambda$, where $\Lambda$ is a subset of
    $\tfrac{\pi c}{L} \IN$.
\end{lema}
This lemma will be proved later in Section~\ref{sec:uniqueness_limit}.

We give now the main theoretical result of this paper.
\begin{theo}[Weak convergence to the limit problem]
    \label{theo:limit_term}
    Let $\omega \not\in\Lambda$ with the set $\Lambda$ in Lemma~\ref{lema:well-posedness_limit},
    $\bff \in \Hs(\Div,\Omega)$ and there exist two constants $C_\Omega > 0$ and $\delta_0 > 0$ such that 
    for all $\delta \in (0, \delta_0)$ it holds for the solution 
    $(\bv^\delta, p^\delta)$ of~\eqref{eq:Navier_Stokes}
    \begin{equation}
    \label{eq:universal_estimate}
    \xnorm{\bv^\delta}{\Hs(\Div, \Omega^\delta)} + \delta^2 \xnorm{\Curl
        \bv^\delta}{\Ltwo(\Omega^\delta)^3} +
    \xnorm{p^\delta}{\Hone(\Omega^\delta)}\leqslant C_\Omega.
    \end{equation}
    Then $(\bv^\delta,p^\delta)$ converges weakly in $\Hs(\Div, \Omega)
    \times \Hone(\Omega)$ to the solution $(\bv_0,p_0)$ of~\eqref{eq:limit_term_complete}.
\end{theo}
In a forthcoming article
we shall prove the estimate~\eqref{eq:universal_estimate}.

\begin{figure}[bt]
    \centering
    \null\hfill
    \includegraphics[height=4cm]{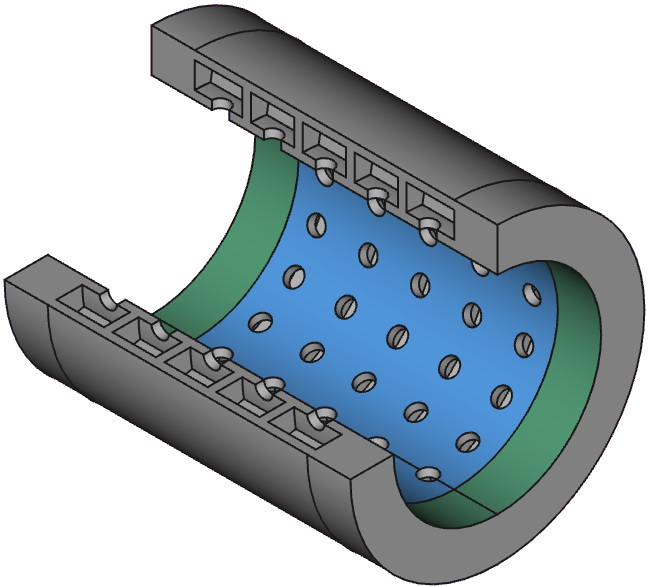}
    \hfill\null
    
    \caption{A cylindrical liner for which in the limit $\delta \to 0$ the same
        impedance boundary conditions appear as for the considered flat surface.}
    \label{fig:liners}
\end{figure}

\begin{rem}
    The following study is done on a flat interface $\Gamma$ for simplicity. For
    slow varying interfaces, the upcoming limit model can be derived
    using an appropriate variable change that flattens the surface. For
    the example of a cylindrical array of Helmholtz resonators, as it
    can be seen on Fig.~\ref{fig:liners}, such a
    variable change has been used in a previous work\cite{Schmidt.Semin.ThoensZueva.Bake:2018}.
\end{rem}

\section{Proof of the weak convergence to the limit}
\label{sec:deriv-just-main}

In this section, we derive the limit problem~(\ref{eq:limit_term_complete}) on
$\Omega$ using the two-scale convergence. To do so, we show that, up to a
subsequence, $(\bv^\delta,p^\delta)$ converges weakly in
$\Hs(\Div, \Omega) \times \Hone(\Omega)$ to a limit $(\bv_0,p_0)$ that
satisfies an Helmholtz-$\nabla \Div$ equation with radiation conditions. To
obtain this result we prove the weak convergence of $\bv^\delta$ in each
subpart of $\Omega^\delta$, \ie the domain $\Omega$ not including the interface
$\Gamma$, the array of resonator chambers, the two-semi infinites strips of the
pattern $\mB^\delta$ and the array of resonator apertures, where matching
conditions and finally impedance conditions follow.

From this stability result, we defive another \apriori error estimate on what
we will call the \emph{extended Helmholtz resonator array}. We first extend the
Helmholtz resonator $\Omega^\delta_H(\bx^\delta_\Gamma)$ centered at
$\bx^\delta_\Gamma$ into an \emph{extended Helmholtz resonator}
$\tilde{\Omega}^\delta_H(\bx^\delta_\Gamma)$ defined as
\begin{equation*}
\tilde{\Omega}^\delta_H(\bx^\delta_\Gamma) =
\Omega^\delta_H(\bx^\delta_\Gamma) \cup \big( \bx^\delta_\Gamma +
\mA \times (0, 2 \sqrt{\delta}) \big).
\end{equation*}

\begin{lema}
    \label{lema:better_estimate}
    Let $\omega \not\in \Lambda$, and let the assumption estimate~(\ref{eq:universal_estimate}) holds. 
    There exists two constants $C_H > 0$ and $\delta_0 > 0$ such that, for any
    $\delta \in (0, \delta_0)$, the estimate
    \begin{multline}
    \label{eq:lema_better_estimate}
    \sum_{\bx^\delta_\Gamma \in \Gamma^\delta}\xnorm{\bv^\delta}{\Hs(\Div,
        \tilde{\Omega}^\delta_H(\bx^\delta_\Gamma))}^2 + \sum_{\bx^\delta_\Gamma
        \in \Gamma^\delta} \delta^2 \xnorm{\Curl 
        \bv^\delta}{\Ltwo(\tilde{\Omega}^\delta_H(\bx^\delta_\Gamma))^3}^2 \\ +
    \sum_{\bx^\delta_\Gamma \in \Gamma^\delta}
    \xnorm{p^\delta}{\Hone(\tilde{\Omega}^\delta_H(\bx^\delta_\Gamma))}^2 
    \leqslant C_H.
    \end{multline}
    holds.
\end{lema}

\begin{proof}
    Let $\omega \not\in \Lambda$, and let the assumption estimate~(\ref{eq:universal_estimate}) holds. 
    Taking the square of this estimate, and using that the union of all $\tilde{\Omega}^\delta_H(\bx^\delta_\Gamma)$ over $\bx^\delta_\Gamma$
    is a subset of $\Omega^\delta$ leads to estimate~(\ref{eq:lema_better_estimate}) with $C_H = 3 C_\Omega^2$. 
\end{proof}

In the following, we use the estimate of this Lemma in each subsection: the array of Helmholtz
resonators, the array of patterns below apertures, the array of apertures and
the array of patterns above apertures.

\subsection{Weak convergence in the resonator array}
\label{sec:weak-conv-reson}

In this section, we consider for each $\delta > 0$,
$\bx = (x_1,x_2,x_3) \in \Omega^\delta$ with $x_3 < - \sqrt{\delta}$. At the
first glance, we have to study positions $\bx$ depending on $\delta$, since the
location of each Helmholtz resonator depends on $\delta$. However, later we
will see how to separate these dependencies.

Due to the geometrical assumption on the array of Helmholtz resonators, for
each $\bx$ in the array of Helmholtz resonators, there exists a center of
aperture of one resonator, which we call resonator position
$\bx^\delta_\Gamma \in \Gamma^\delta$ such that
$\bx \in \Omega^\delta_C(\bx^\delta_\Gamma)$. We introduce the two-dimensional
point $\by \in \mA_C$, that we identify with abuse of notation to the
three-dimensional point $(\by,0)$, such that
\begin{equation*}
\bx \mapsto (\by,x_3) := \big(\tfrac{1}{\delta} \big( \bx - \bx^\delta_\Gamma
\big) , x_3 \big) \in \mA_C \times (-L,-\sqrt{\delta}),
\end{equation*}
\ie $\bx = \bx^\delta_\Gamma + (0,0,x_3) + \delta \by$. Using the coordinate
$(\by,x_3)$ means to stretch the resonator
$\Omega^\delta_H(\bx^\delta_\Gamma)$in the transverse plane $(\be_1,\be_2)$,
and in this plane only. The stretched resonator chamber is denoted by
$\widehat{\Omega}^\delta_C := \mA_C \times (-L, -\sqrt{\delta})$.

In the following, we introduce the five-dimensional functions
$\bV^\delta$ and $P^\delta$, depending on the resonator position
$\bx^\delta_\Gamma$, the slow longitudinal variable $x_3$ and the fast
transverse variable~$\by$, by
\begin{equation}
\label{eq:change_function_resonator}
\begin{aligned}
\bV^\delta(\bx^\delta_\Gamma,x_3,\by) & = \bv^\delta(\bx^\delta_\Gamma
+ (0,0,x_3) + \delta \by), \\
P^\delta(\bx^\delta_\Gamma,x_3,\by) & = p^\delta(\bx^\delta_\Gamma
+ (0,0,x_3) + \delta \by).
\end{aligned}
\end{equation}

Considering the linearized Navier-Stokes problem~\eqref{eq:Navier_Stokes} with
the Laplace operator written as $\Delta = \nabla \Div - \Curl \Curl$ and
applying the anisotropic coordinate change, we obtain the system
\begin{subequations}
    \label{eq:Navier_Stokes:resonator}
    \begin{align}
    \label{eq:Navier_Stokes:resonator:M}
    - \imath \omega \bV^\delta + \tfrac{1}{\delta \rho_0} \nabla_{\by}
    P^\delta + \tfrac{1}{\rho_0}
    \partial_{x_3} P^\delta \be_3
    & \\ \nonumber - (\nu_0 + \nu_0') \delta^2 \nabla_{\by}
    \Div_\by \bV^\delta  & \\ \nonumber - (\nu_0 + \nu_0') \delta^2
    \partial_{\be_3} \Div_\by \bV^\delta \be_3 - & \\ \nonumber
    (\nu_0 + \nu_0') \delta^3 \nabla_{\by} \big( \partial_{x_3} \bV^\delta
    \cdot \be_3 \big) & \\ \nonumber - (\nu_0 + \nu_0') \delta^4 \partial_{x_3} \big(
    \partial_{x_3} \bV^\delta \cdot \be_3
    \big) & \\ \nonumber + \nu_0 \delta^4 \big( \tfrac{1}{\delta} \partial_{y_1},
    \tfrac{1}{\delta} \partial_{y_2}, \partial_{x_3} \big) \wedge
    \Big( \big( \tfrac{1}{\delta} \partial_{y_1}, 
    \tfrac{1}{\delta} \partial_{y_2}, \partial_{x_3} \big) \wedge 
    \bV^\delta \Big)
    &= \zerobf, \quad \text{in }\Gamma^\delta \times \widehat{\Omega}^\delta_C, 
    \\ 
    - \imath \omega P^\delta + \tfrac{\rho_0
        c^2}{\delta} \Div_{\by} \bV^\delta + 
    \rho_0 c^2 \big( \partial_{x_3}
    \bV^\delta \cdot \be_3 \big) 
    &= 0, \quad \text{in
    }\Gamma^\delta \times
    \widehat{\Omega}^\delta_C, \label{eq:Navier_Stokes:resonator:C}
    \\ 
    \bV^\delta
    &= \zerobf , \quad \text{on } \Gamma^\delta \times
    (-L,-\sqrt{\delta}) \times \partial \mA_C
    , \label{eq:Navier_Stokes:resonator:B1} \\  
    \bV^\delta
    &= \zerobf , \quad \text{on }
    \Gamma^\delta \times \lbrace -L
    \rbrace \times \mA_C . \label{eq:Navier_Stokes:resonator:B2}
    \end{align}
\end{subequations}

The estimate~(\ref{eq:lema_better_estimate}) of the
Lemma~\ref{lema:better_estimate} is equivalent to state in rescaled coordinates
that
\begin{multline}
\label{eq:lema_better_estimate_scaled}
\sum_{\bx^\delta_\Gamma \in \Gamma^\delta} \xnorm{\bV^\delta(\bx_\Gamma^\delta,\cdot)}{\Ltwo(\widehat{\Omega}^\delta_C)^3}^2 
+ \tfrac{1}{\delta^2}  \sum_{\bx^\delta_\Gamma \in \Gamma^\delta}
\xnorm{\Div_\by 
    \bV^\delta}{\Ltwo(\widehat{\Omega}^\delta_C)}^2  
\\ + \sum_{\bx^\delta_\Gamma \in \Gamma^\delta} \xnorm{
    \partial_{x_3} \bV^\delta(\bx_\Gamma^\delta,\cdot) \cdot
    \be_3}{\Ltwo(\widehat{\Omega}^\delta_C)}^2 
+ \delta^4 \sum_{\bx^\delta_\Gamma \in \Gamma^\delta} \xnorm{\big( \tfrac{1}{\delta} \partial_{y_1}, 
    \tfrac{1}{\delta} \partial_{y_2}, \partial_{x_3} \big) \wedge 
    \bV^\delta(\bx^\delta_\Gamma,\cdot)}{\Ltwo(\widehat{\Omega}^\delta_C)^3}^2
\\ + \sum_{\bx^\delta_\Gamma \in \Gamma^\delta} \xnorm{
    P^\delta(\bx^\delta_\Gamma,\cdot)}{\Ltwo(\widehat{\Omega}^\delta_C)}^2 +
\tfrac{1}{\delta^2} \sum_{\bx^\delta_\Gamma \in \Gamma^\delta} \xnorm{
    \nabla_\by p^\delta(\bx^\delta_\Gamma,\cdot)}{\Ltwo(\tilde{\Omega}^\delta_H(\bx^\delta_\Gamma))^2}^2 
\\ + \sum_{\bx^\delta_\Gamma \in \Gamma^\delta} \xnorm{
    \partial_{x_3} P^\delta(\bx^\delta_\Gamma,\cdot)}{\Ltwo(\widehat{\Omega}^\delta_C)}^2 \leqslant C_H.
\end{multline}

The main idea is to extend the functions $(\bV^\delta,P^\delta)$
\emph{discrete} with respect to $\bx^\delta_\Gamma \in \Gamma^\delta$ to
functions still denoted by $(\bV^\delta,P^\delta)$ and \emph{continuous} with
respect to $\bx \in \Gamma$ by the following
\begin{equation*}
(\bV^\delta,P^\delta)(\bx_\Gamma,\cdot) =
(\bV^\delta,P^\delta)(\bx^\delta_\Gamma, \cdot), \quad \bx_\Gamma \in
\bx^\delta_\Gamma + \delta \mA.
\end{equation*}
The equation~(\ref{eq:Navier_Stokes:resonator}) is then extended naturally on
$\bx_\Gamma \in \Gamma$ (the parameter $\bx^\delta_\Gamma$ is only playing a
parameter), and the \emph{discrete} error
estimate~(\ref{eq:lema_better_estimate_scaled}) is extended to a
\emph{continuous} error estimate
\begin{multline}
\label{eq:lemma_better_estimate_scaled_continuous}
\xnorm{\bV^\delta}{\Ltwo(\Gamma;
    \Ltwo(\widehat{\Omega}^\delta_C)^3)}^2 + \tfrac{1}{\delta^2}
\xnorm{\Div_{\by} \bV^\delta}{\Ltwo(\Gamma;
    \Ltwo(\widehat{\Omega}^\delta_C))}^2 + \xnorm{\partial_{x_3}
    \bV^\delta \cdot \be_3}{\Ltwo(\Gamma;
    \Ltwo(\widehat{\Omega}^\delta_C))}^2 \\ + \delta^4
\xnorm{\big( \tfrac{1}{\delta} \partial_{y_1}, \tfrac{1}{\delta}
    \partial_{y_2}, \partial_{x_3} \big) \wedge \bV^\delta}{\Ltwo(\Gamma;
    \Ltwo(\widehat{\Omega}^\delta_C)^3)}^2 +
\xnorm{P^\delta}{\Ltwo(\Gamma; \Ltwo(\widehat{\Omega}^\delta_C))}^2 \\
+ \frac{1}{\delta^2} \xnorm{\nabla_\by P^\delta}{\Ltwo(\Gamma;
    \Ltwo(\widehat{\Omega}^\delta_C)^2)}^2 + \xnorm{\partial_{x_3}
    P^\delta}{\Ltwo(\Gamma; \Ltwo(\widehat{\Omega}^\delta_C))}^2 \leqslant
C_H |\mA|.
\end{multline}

This means that the sequence $(\bV^\delta,P^\delta)$ is bounded in
$\Ltwo(\Gamma; \Hs(\Div, \widehat{\Omega}^\delta_C)) \times \Ltwo(\Gamma;
\Hone(\widehat{\Omega}^\delta_C))$,
independent of $\delta \to 0$ even the domain $\widehat{\Omega}^\delta_C$
enlarges for decreasing $\delta$.  Then, for any fixed $\epsi > 0$ and for any
$\delta < \epsi$, the sequence $(\bV^\delta,P^\delta)$ is bounded in
$\Ltwo(\Gamma; \Hs(\Div, \widehat{\Omega}^\epsi_C)) \times \Ltwo(\Gamma;
\Hone(\widehat{\Omega}^\epsi_C))$
therefore we can extract a subsequence that we still denote by
$(\bV^\delta, P^\delta)$ that converges to a limit $(\bV_0^\epsi, P_0^\epsi)$
weakly in
$\Ltwo(\Gamma; \Hs(\Div, \widehat{\Omega}^\epsi_C)) \times \Ltwo(\Gamma;
\Hone(\widehat{\Omega}^\epsi_C))$.
Combining the lower semi-continuity of the weak
limit stated by the Theorem 2.2.1 of the book of Evans\cite{evans1990weak} with the
estimate~\eqref{eq:lemma_better_estimate_scaled_continuous}, we find
\begin{equation*}
\xnorm{\nabla_\by P_0^\epsi}{\Ltwo(\Gamma;
    \Ltwo(\widehat{\Omega}^\epsi_C)^2)} \leqslant 
\liminf_{\delta \to 0} \xnorm{\nabla_\by
    P^\delta}{\Ltwo(\Gamma;
    \Ltwo(\widehat{\Omega}^\delta_C)^2)} \leqslant
\liminf_{\delta \to 0} \delta \sqrt{C_H |\mA|} = 0.
\end{equation*}
This motivates us to take the scalar product
of~\eqref{eq:Navier_Stokes:resonator} with test functions $(\bW,Q)$
with $Q$ independent of the fast transverse variable $\by$. Then,
integrating by parts equation~\eqref{eq:Navier_Stokes:resonator} leads
to
\begin{equation}
\label{eq:Navier_Stokes_resonator:bv}
\begin{aligned}
- \imath \omega
\xscal{\bV^\delta(\bx_\Gamma,\cdot)}{\bW}{(\widehat{\Omega}^\epsi_C)^3}
+ \tfrac{1}{\rho_0} \xscal{\partial_{x_3}
    P^\delta(\bx_\Gamma,\cdot)
    \be_3}{\bW}{(\widehat{\Omega}^\epsi_C)^3} & \\ - \nu_0 \delta^4
\xscal{\big( \tfrac{1}{\delta} \partial_{y_1}, \tfrac{1}{\delta}
    \partial_{y_2}, \partial_{x_3} \big) \wedge
    \bV^\delta(\bx_\Gamma,\cdot)}{( 0, 0, \partial_{x_3} ) \wedge
    \bW}{(\widehat{\Omega}^\epsi_C)^3} & \\ + (\nu_0 + \nu'_0)
\delta^2 \xscal{\Div_\by(\bx_\Gamma,\cdot)
    \bV^\delta}{\Div_\by \bW \cdot
    \be_3}{\widehat{\Omega}^\epsi_C} \\ + (\nu_0 + \nu'_0)
\delta^3 \xscal{\partial_{x_3}(\bx_\Gamma,\cdot)
    \bV^\delta}{\Div_\by \bW \cdot
    \be_3}{\widehat{\Omega}^\epsi_C} \\ + (\nu_0 + \nu'_0)
\delta^3 \xscal{\Div_\by(\bx_\Gamma,\cdot)
    \bV^\delta}{\partial_{x_3} \bW \cdot
    \be_3}{\widehat{\Omega}^\epsi_C} & \\+ (\nu_0 + \nu'_0) \delta^4
\xscal{\partial_{x_3} \bV^\delta(\bx_\Gamma,\cdot) \cdot
    \be_3}{\partial_{x_3} \bW \cdot \be_3}{\widehat{\Omega}^\epsi_C}
& = 0, \\
- \imath \omega
\xscal{P^\delta(\bx_\Gamma,\cdot)}{Q}{\widehat{\Omega}^\epsi_C} +
\rho_0 c^2 \xscal{\nabla_\by
    \bV^\delta(\bx_\Gamma,\cdot)}{Q}{\widehat{\Omega}^\epsi_C} & \\ + \rho_0
c^2 \xscal{\partial_{x_3} \bV^\delta(\bx_\Gamma,\cdot) \cdot
    \be_3}{Q}{\widehat{\Omega}^\epsi_C} & = 0.
\end{aligned}
\end{equation}
Then, using the weak convergence of $(\bV^\delta,P^\delta)$ to
$(\bV_0^\epsi, P_0^\epsi)$ in
$\Ltwo(\Gamma; \Hs(\Div, \widehat{\Omega}^\epsi_C)) \times \Ltwo(\Gamma;
\Hone(\widehat{\Omega}^\epsi_C))$
using estimate~(\ref{eq:lemma_better_estimate_scaled_continuous}), we find the
limit $\delta \to 0$ for almost all $\bx_\Gamma \in \Gamma$
\begin{subequations}
    \label{eq:Navier_Stokes_resonator:bv_lim}
    \begin{align}
    \label{eq:Navier_Stokes_resonator:bv_lim:M}
    - \imath \omega
    \xscal{\bV_0^\epsi(\bx_\Gamma,\cdot)}{\bW}{(\widehat{\Omega}^\epsi_C)^3}
    + \tfrac{1}{\rho_0} \xscal{\partial_{x_3}
        P_0^\epsi(\bx_\Gamma,\cdot)
        \be_3}{\bW}{(\widehat{\Omega}^\epsi_C)^3} & =
    0, \\
    \label{eq:Navier_Stokes_resonator:bv_lim:C}
    - \imath \omega
    \xscal{P_0^\epsi(\bx_\Gamma,\cdot)}{Q}{\widehat{\Omega}^\epsi_C} +
    \rho_0 c^2 \xscal{\partial_{x_3} \bV_0^\epsi(\bx_\Gamma,\cdot)
        \cdot \be_3}{Q}{\widehat{\Omega}^\epsi_C} & = 0.
    \end{align}
\end{subequations}
The arbitrary choice of $\bW$ leads to
\begin{equation*}
- \imath \omega \bV_0^\epsi (\bx_\Gamma,\cdot) +
\tfrac{1}{\rho_0} \partial_{x_3} P_0^\epsi (\bx_\Gamma,\cdot) \be_3 = 0,
\end{equation*}
and gives \aposteriori that $\bV_0^\epsi(\bx_\Gamma,\cdot)$ is also independent
of $\by$ and is directed among the $\be_3$ direction, \ie there exists a scalar
function $V_{0,3}^\epsi$ independent of $\by$ such that
\begin{equation*}
\bV_0^\epsi(\bx_\Gamma,x_3,\by) = V_{0,3}^\epsi(\bx_\Gamma,x_3) \be_3.
\end{equation*}
Taking then the equation~(\ref{eq:Navier_Stokes_resonator:bv_lim:C}), we deduce
that
\begin{equation*}
- \imath \omega P_0^\epsi (\bx_\Gamma,\cdot) + \rho_0 c^2 \partial_{x_3}
\bV_0^\epsi(\bx_\Gamma,\cdot) \cdot \be_3 = 0,
\end{equation*}
\ie that $(\bV_0^\epsi(\bx_\Gamma,\cdot),P_0^\epsi(\bx_\Gamma,\cdot))$ is
solution of an homogeneous one-dimensional Helmholtz equation. Derivation of
the boundary condition at $x_3 = -L$ is done as follow: we take a particular
test function $W$ depending on $(x_3,\by) \in (-L,0) \times \mA_C$ such that
$W(x_3,\by) = 1$ for $x_3 < -3L/4$ and $W(x_3,\by) = 0$ for $x_3 > -L/2$, and
using a one-dimensional Stokes formula coupled to the weak convergence of
$\bV^\delta$ to $V^\epsi_{0,3} \be_3$ in
$(\Ltwo(\Gamma; \Ltwo(\widehat{\Omega}^\delta_C))$ gives
\begin{multline*}
0 = \lim_{\delta \to 0} \int_{\mA_C} \Big\lbrace \partial_{x_3}
(\bV^\delta(\bx_\Gamma,x_3,\by) \cdot \be_3 -
V^\epsi_{0,3}(\bx_\Gamma,x_3)) W(x_3) \\ +
(\bV^\delta(\bx_\Gamma,x_3,\by) \cdot \be_3 -
V^\epsi_{0,3}(\bx_\Gamma,x_3)) \partial_{x_3} W(x_3) \Big\rbrace
\,\mathrm{d}\by \,\mathrm{d}x_3 = - a_C V^\epsi_{0,3}(\bx_\Gamma,x_3),
\end{multline*}
so there exists a scalar function $V^\epsi_0$ depending only on the resonator
position such that
\begin{equation}
\label{eq:expression_resonator_epsi}
\begin{aligned}
\bV^\epsi_0(\bx_\Gamma,x_3) & = V^\epsi_0(\bx_\Gamma) \sin \big(
\tfrac{\omega}{c} (L+x_3) \big) \\
P^\epsi_0(\bx_\Gamma,x_3)&  = - \imath \rho_0 c V^\epsi_0(\bx_\Gamma) \cos \big(
\tfrac{\omega}{c} (L+x_3) \big).
\end{aligned}
\end{equation}
The next point is to derive a limit problem on the domain
$\widehat{\Omega}_C := \lim_{\epsi \to 0} \widehat{\Omega}_C^\epsi =
\mA_C \times (-L,0)$. To do so, using the lower semi-continuity of the
weak limit $(\bV_0^\epsi,P^\epsi_0)$ with the \apriori
estimate~(\ref{eq:lemma_better_estimate_scaled_continuous}), it holds 
\begin{multline*}
\xnorm{\bV_0^\epsi}{\Ltwo(\Gamma,\Ltwo(\widehat{\Omega}^\epsi_C))}^2
+ \xnorm{\partial_{x_3}
    \bV_0^\epsi}{\Ltwo(\Gamma,\Ltwo(\widehat{\Omega}^\epsi_C))}^2 
\\ + \xnorm{P_0^\epsi}{\Ltwo(\Gamma,\Ltwo(\widehat{\Omega}^\epsi_C))}^2
+ \xnorm{\partial_{x_3}
    P_0^\epsi}{\Ltwo(\Gamma,\Ltwo(\widehat{\Omega}^\epsi_C))}^2 
\leqslant C_H |\mA|.
\end{multline*}
This error estimate combined to~(\ref{eq:expression_resonator_epsi}) leads to
\begin{equation*}
\big( 1 + \tfrac{\omega^2}{c^2} + \rho_0^2 c^2 + \rho_0^2 \omega^2 \big)
(L-\sqrt{\epsi}) \xnorm{V^\epsi_0}{\Ltwo(\Gamma)}^2 \leqslant 2 C_H |\mA|.
\end{equation*}
Extending $(\bV^\epsi_0,P^\epsi_0)$ on $\mA \times (-L,0)$ using
formula~(\ref{eq:expression_resonator_epsi}) and using this last error
estimate leads to, for $2\sqrt{\epsi} < L$,
\begin{multline*}
\xnorm{\bV_0^\epsi}{\Ltwo(\Gamma,\Ltwo(\widehat{\Omega}^0_C))}^2
+ \xnorm{\partial_{x_3}
    \bV_0^\epsi}{\Ltwo(\Gamma,\Ltwo(\widehat{\Omega}^0_C))}^2 
\\ + \xnorm{P_0^\epsi}{\Ltwo(\Gamma,\Ltwo(\widehat{\Omega}^0_C))}^2
+ \xnorm{\partial_{x_3}
    P_0^\epsi}{\Ltwo(\Gamma,\Ltwo(\widehat{\Omega}^0_C))}^2 
\leqslant \tfrac{2L}{(L-\sqrt{\epsi})} C_H |\mA| \leqslant 4 C_H |\mA|. 
\end{multline*}
so that $(\bV^\epsi_0,P^\epsi_0)$ is uniformly bounded in $\Ltwo(\Gamma;\Hs(\Div;
\widehat{\Omega}_C) \times \Hone(\widehat{\Omega}_C))$. There exists then a
subsequence that we still denote by $(\bV^\epsi_0,P^\epsi_0)$ that weakly
converges to a limit $(\bV_0,P_0) \in \Ltwo(\Gamma;\Hs(\Div;
\widehat{\Omega}_C) \times \Hone(\widehat{\Omega}_C))$.

We have two small quantities, $\delta$ and $\epsi$, and we want to make these
two quantities tending to $0$. To do so, we will make a diagonal construction
of $(\bV^\delta, P^\delta)$ to $(V_{0,3} \be_3, P_0)$ by the following
\begin{enumerate}
    \item $(\bV^\delta, P^\delta)$ is bounded in
    $\Ltwo\big( \Gamma; (\Hs(\Div;\widehat{\Omega}^{\epsi_1}_C))^3 \times
    \Hone(\widehat{\Omega}^{\epsi_1}_C) \big)$
    using the energy estimate~\eqref{eq:lemma_better_estimate_scaled_continuous}
    for any $\delta \leqslant \epsi_1$, so there exists a sequence
    $(\bV^{\delta_n^{(1)}}, P^{\delta_n^{(1)}})_{n \in \IN}$ that weakly converges to
    $(\bV^{\epsi_1}, P^{\epsi_1})$
    $\Ltwo\big( \Gamma; (\Hs(\Div;\widehat{\Omega}^{\epsi_1}_C)) \times
    \Hone(\widehat{\Omega}^{\epsi_1}_C) \big)$. We consider the sequence of
    decreasing indices $(\delta_n^{(1)})_{n \in \IN}$ such that $\delta_0^{(1)}
    \leqslant \epsi_2$,
    \item $(\bV^{\delta_n^{(1)}}, P^{\delta_n^{(1)}})_{n \in \IN}$ is bounded in
    $\Ltwo\big( \Gamma; (\Hs(\Div;\widehat{\Omega}^{\epsi_2}_C)) \times
    \Hone(\widehat{\Omega}^{\epsi_2}_C) \big)$
    using the energy estimate~\eqref{eq:lemma_better_estimate_scaled_continuous}
    since $\delta_n^{(1)} \leqslant \delta_0^{(1)} \leqslant \epsi_2$ for any
    $n \in \IN$, so there exists a subsequence
    $(\bV^{\delta_n^{(2)}}, P^{\delta_n^{(2)}})$ that weakly converges to
    $(\bV^{\epsi_2}, P^{\epsi_2})$
    $\Ltwo\big( \Gamma; (\Hs(\Div;\widehat{\Omega}^{\epsi_2}_C)) \times
    \Hone(\widehat{\Omega}^{\epsi_2}_C) \big)$.
    We consider the sequence of decreasing indices $(\delta_n^{(2)})_{n \in \IN}$
    such that $\delta_0^{(2)} \leqslant \epsi_3$. Finally due to the extraction
    process, it holds that
    $(\bV^{\epsi_2},P^{\epsi_2})=(\bV^{\epsi_1},P^{\epsi_1})$ on
    $\widehat{\Omega}^{\epsi_1}_C$,
    \item iteratively for any $k \geqslant 1$,
    $(\bV^{\delta_n^{(k)}}, P^{\delta_n^{(k)}})_{n \in \IN}$ is bounded in
    $\Ltwo\big( \Gamma; (\Hs(\Div;\widehat{\Omega}^{\epsi_{k+1}}_C)) \times
    \Hone(\widehat{\Omega}^{\epsi_{k+1}}_C) \big)$
    using the energy estimate~\eqref{eq:lemma_better_estimate_scaled_continuous} since
    $\delta_n^{(k)} \leqslant \delta_0^{(k)} \leqslant \epsi_{k+1}$ for any
    $n \in \IN$, so there exists a subsequence
    $(\bV^{\delta_n^{(k+1)}}, P^{\delta_n^{(k+1)}})$ that weakly converges to
    $(\bV^{\epsi_{k+1}}, P^{\epsi_{k+1}})$
    $\Ltwo\big( \Gamma; (\Hs(\Div;\widehat{\Omega}^{\epsi_{k+1}}_C))^3 \times
    \Hone(\widehat{\Omega}^{\epsi_{k+1}}_C) \big)$.
    We consider the sequence of decreasing indices
    $(\delta_n^{(k+1)})_{n \in \IN}$ such that
    $\delta_0^{(k+1)} \leqslant \epsi_{k+2}$. Finally due to the extraction
    process, it holds that
    $(\bV^{\epsi_{k+1}},P^{\epsi_{k+1}})=(\bV^{\epsi_k},P^{\epsi_k})$ on
    $\widehat{\Omega}^{\epsi_k}_C$.
\end{enumerate}
We finally take the sequence $(\bV^{\delta_n^{(n)}},P^{\delta_n^{(n)}})$. 
By property of the intermediate sequence $(\bV^{\epsi_n}_0,P^{\epsi_n}_0)$,
$(\bV^{\delta_n^{(n)}},P^{\delta_n^{(n)}})$ weakly converges to $(V_0,P_0)$ in
$\Ltwo(\Gamma;\Hs(\Div,K)) \times \Ltwo(\Gamma;\Hone(K))$ for any $K \subset
\widehat{\Omega}_C$ whose minimal to the boundary $\lbrace x_3 = 0 \rbrace$ is
positive. Using~(\ref{eq:expression_resonator_epsi}), we obtain the following
\begin{prop}
    There exists a scalar function $V_0 \in \Ltwo(\Gamma)$ such that
    \begin{equation}
    \label{eq:one_dimensional_expression}
    \begin{aligned}
    \bV_0(\bx_\Gamma,x_3) & = V_0(\bx_\Gamma) \sin \big( \tfrac{\omega}{c}
    (x_3+L) \big) \,
    \be_3, \\ P_0(\bx_\Gamma,x_3) & = - \imath \rho_0 c
    V_0(\bx_\Gamma) \cos \big( \tfrac{\omega}{c} (x_3+L) \big).
    \end{aligned}
    \end{equation}
\end{prop}

\textbf{Conclusion:}
Using the weak convergence of $(\bV^\delta,P^\delta)$ to $(\bV_0,P_0)$
and~(\ref{eq:one_dimensional_expression}), we get the following limit interface
conditions: 
\begin{equation}
\label{eq:one_dimensional_interface_conditions}
\begin{aligned}
\lim_{\delta \to 0} \int_{\mA_C} P^\delta\big(\bx_\Gamma,-2
\sqrt{\delta},(y_1,y_2)\big) \, \mathrm{d}y_1 \, \mathrm{d}y_2 & = -a_C
\imath \rho_0 c V_0(\bx_\Gamma) \cos \big( \tfrac{\omega L}{c} \big), \\
\lim_{\delta \to 0} \int_{\mA_C} \bV^\delta \big(\bx_\Gamma,-2
\sqrt{\delta},(y_1,y_2)\big) \cdot \be_3 \, \mathrm{d}y_1 \, \mathrm{d}y_2
& = a_C V_0(\bx_\Gamma) \sin \big( \tfrac{\omega L}{c} \big).
\end{aligned}
\end{equation}

\subsection{Weak convergence in the pattern below aperture}
\label{sec:weak-conv-patt}

Again we consider for each $\delta > 0$ and any $\bx \in \Omega^\delta
\setminus \Omega$ the unique corresponding resonator position
$\bx^\delta_\Gamma$ with $\bx \in \Omega^\delta_H(\bx^\delta_\Gamma)$. Then we
define the zone below aperture as
\begin{equation*}
\Omega^\delta_-(\bx^\delta_\Gamma) := \left\lbrace \bx \in \bx^\delta_\Gamma
+ \delta \mA_C \times (-2\sqrt{\delta},-h_0
\delta^2) \text{ such that } \bx \not\in B_3(\bx^\delta_\Gamma-\delta h_0
\be_3, \delta^{\tfrac{3}{2}}) \right\rbrace,
\end{equation*}
and the rescaled zone below aperture (see Fig.~\ref{fig:domain_aperture_minus})
for the variable change $\by = \delta^{-1} (\bx-\bx^\delta_\Gamma)$
\begin{equation*}
\mB^\delta_- := \left\lbrace \by \in \mA_C \times (-2/\sqrt{\delta},-h_0
\delta) \text{ such that } \by \not\in B_3(-\delta h_0 \be_3, \sqrt{\delta}) \right\rbrace,
\end{equation*}
where $B_d(\bx_0,r)$ denotes the $d$-dimensional ball centered at $\bx_0$ and of
radius $r$.

\begin{figure}[!bt]
    \centering
    \begin{tikzpicture}[scale=0.6]
    \draw[white] (-2.5,-3) rectangle (2.5,0.5);
    \fill[black!20!white] (-3,-7) rectangle (3,-0.345);
    \filldraw[draw=red, fill=white] (-1,-0.345) arc(180:360:1) node
    [pos=0.5,below] {$\color{red} \Gamma^\delta_-(\sqrt{\delta})$};
    \draw[thick] (-3.5,0.345) -- (-0.346,0.345) -- (-0.346,-0.345) --
    (-1,-0.345);
    \draw[thick] (-3,-0.345) -- (-3,-8);
    \draw[thick] (3.5,0.345) -- (0.346,0.345) -- (0.346,-0.345) --
    (1,-0.345);
    \draw[thick] (3,-0.345) -- (3,-8);
    \draw[blue,thick] (-3,-0.345) -- (-1,-0.345);
    \draw[blue,thick] (1,-0.345) -- (3,-0.345) node [pos=1.0,right] {$\mG^\delta_-$};
    \draw[->] (0,-0.345) ++ (0,0) -- ++ (300:1) node [pos=1.1,right]
    {$\sqrt{\delta}$}; 
    \draw[<->] (1.8,-0.345) -- (1.8,0.345) node[pos=0.5, right]
    {$h_0 \delta$};
    \draw[<->] (-0.346,-0.28) -- (0.346,-0.28) node[pos=0.5, below]
    {$d_0 \delta$};
    \draw (-3,-7) -- (3,-7) node[pos=0.5,below] {$y_3 = - 2 / \sqrt{\delta}$};
    \filldraw (0,0.345) circle (0.1) node [above] {$\zerobf$};
    \end{tikzpicture}
    \caption{Illustration of the canonical domain $\mB^\delta_-$ for the pattern
        below aperture (gray).}
    \label{fig:domain_aperture_minus}
\end{figure}
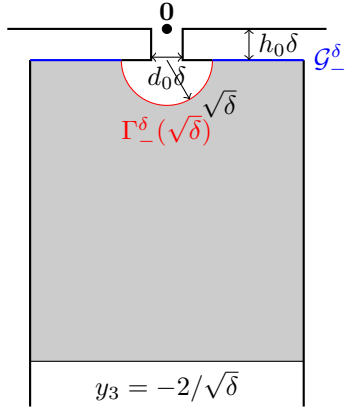

Here, we are interested in the behaviour of the solution
$(\bv^\delta,p^\delta)$ solution of~(\ref{eq:Navier_Stokes}) in the resonator array, close to the
apertures. To do so, we introduce the
five-dimensional functions $\bPsi^\delta_-$ and $\Phi^\delta_-$ by
\begin{equation}
\label{eq:change_function_aperture_below}
\bPsi^\delta_-(\bx^\delta_\Gamma, \by) = \bv^\delta(\bx^\delta_\Gamma
+ \delta \by) \quad \text{and} \quad 
\Phi^\delta_-(\bx^\delta_\Gamma, 
\by) = p^\delta(\bx^\delta_\Gamma + \delta \by).
\end{equation}
Then, in view of estimate~\eqref{eq:lema_better_estimate} of
Lemma~\ref{lema:better_estimate}, the rescaled functions
$(\bPsi^\delta_-,\Phi^\delta_-)$ satisfy the following estimate for
almost all resonators
\begin{multline}
\label{eq:lemma_better_estimate_scaled_pattern-}
\delta
\xnorm{\bPsi^\delta_-(\bx^\delta_\Gamma,\cdot)}{\Ltwo(\mB^\delta_-)}^2
+ \tfrac{1}{\delta} \xnorm{\Div_{\by}
    \bPsi^\delta_-(\bx^\delta_\Gamma,\cdot)}{\Ltwo(\mB^\delta_-)}^2
+ \delta^3 \xnorm{\Curl_{\by}
    \bPsi^\delta_-(\bx^\delta_\Gamma,\cdot)}{\Ltwo(\mB^\delta_-)}^2
\\ + \delta
\xnorm{\Phi^\delta_-(\bx^\delta_\Gamma,\cdot)}{\Ltwo(\mB^\delta_-)}^2
+ \tfrac{1}{\delta}
\xnorm{\nabla_\by \Phi^\delta_-(\bx^\delta_\Gamma,\cdot)}{\Ltwo(\mB^\delta_-)}^2
\leqslant C_H
\end{multline}
We extend again this \emph{discrete} error estimate into a \emph{continuous}
error estimate on $\Ltwo(\Gamma; \Ltwo(\mB^\delta_-))$. This error estimate
allows us to give some results about the rescaled functions.

\textbf{For the rescaled pressure $\Phi^\delta_-$}, using the
Cauchy-Schwartz inequality and the error
estimate~(\ref{eq:lemma_better_estimate_scaled_pattern-}) leads to
\begin{equation*}
\Big( \int_{\mB^\delta_-} |\nabla_\by \Phi^\delta_-(\bx_\Gamma,\by)| \,
\mathrm{d}\by \Big)^2 \leqslant \frac{2}{\sqrt{\delta}} \int_{\mB^\delta_-}
|\nabla_\by \Phi^\delta_-(\bx_\Gamma,\by)|^2 \,
\mathrm{d}\by \leqslant 2 C_H \sqrt{\delta}
\end{equation*}
and this quantity tends to $0$ as $\delta \to 0$, so that the gradient of
$\Phi^\delta_-(\bx_\Gamma,\cdot)$ tends to $0$ almost everywhere in
$(-\infty,0) \times \mA_C$. It remains to check that the average of
$\Phi^\delta_-$ remains uniformly bounded as well. This can be checked using
again the energy estimate~(\ref{eq:lemma_better_estimate_scaled_pattern-}) and
using that
\begin{equation*}
\delta \xnorm{1}{\Ltwo(\mB^\delta_-)}^2 = \delta \Big( a_C \big(
\tfrac{2}{\sqrt{\delta}} - h_0 \delta \big) -
\tfrac{2\pi}{3}\delta^{\tfrac{3}{2}} \Big)^2.
\end{equation*}
We can then extract a subsequence (that we still denote by $\Phi^\delta_-$)
that converges to a constant. Using
\begin{equation*}
P^\delta (\bx_\Gamma,-\tfrac{2}{\sqrt{\delta}},(y_1,y_2)) =
p^\delta\big(\bx_\Gamma + \delta
\big(y_1,y_2,-\tfrac{2}{\sqrt{\delta}}\big)\big) =
\Phi^\delta_-\big(\bx_\Gamma, \big(y_1,y_2,-\tfrac{2}{\sqrt{\delta}}\big)\big)
\end{equation*}
for almost every $(y_1,y_2) \in \mA_C$ and the first line
of~(\ref{eq:one_dimensional_interface_conditions}), we obtain that
\begin{equation}
\label{eq:limit_Phi_-}
\lim_{\delta \to 0} \Phi^\delta_- = - \imath \rho_0 c V_0(\bx_\Gamma) \cos
\big( \tfrac{\omega L}{c} \big).
\end{equation}

\textbf{Similarly, for the rescaled velocity $\bPsi^\delta_-$}, using the
Cauchy-Schwartz inequality and the error
estimate~(\ref{eq:lemma_better_estimate_scaled_pattern-}) leads to
\begin{equation*}
\Big( \int_{\mB^\delta_-} |\Div_\by \bPsi^\delta_-(\bx_\Gamma,\by)| \,
\mathrm{d}\by \Big)^2 \leqslant \frac{2}{\sqrt{\delta}} \int_{\mB^\delta_-}
|\Div_\by \bPsi^\delta_-(\bx_\Gamma,\by)|^2 \,
\mathrm{d}\by \leqslant 2 C_H \sqrt{\delta},
\end{equation*}
and this quantity tends to $0$ as $\delta$ tends to $0$. Using the Gauss
theorem on $\mB^\delta_-$ gives
\begin{multline*}
\int_{\mB^\delta_-} \Div_\by \bPsi^\delta_-= \int_{\mA_C} \bPsi^\delta_-
\big(\bx_\Gamma,\big(y_1,y_2,-\tfrac{2}{\sqrt{\delta}}\big)\big) \cdot (-\be_3) \,
\mathrm{d}y_1 \, \mathrm{d}y_2 \\ + \int_{\Gamma^\delta_-(\sqrt{\delta})}
\bPsi^\delta_-(\bx_\Gamma, \by) \cdot \bn^\delta_-(\by) \, \mathrm{d}\sigma(\by),
\end{multline*}
where $\bn^\delta_-(\by)$ is the unit outward vector of $\mB^\delta_-$, \ie
\begin{equation*}
\bn^\delta_-(\by) = \frac{(y_1,y_2,-(y_3+\delta h_0))}{|(y_1,y_2,-(y_3+ \delta h_0))|}.
\end{equation*}
Using then
\begin{equation*}
\bV^\delta (\bx_\Gamma,-2\sqrt{\delta},(y_1,y_2)) =
\bv^\delta\big(\bx_\Gamma + \delta
\big(y_1,y_2,-\tfrac{2}{\sqrt{\delta}}\big)\big) =
\bPsi^\delta_-\big(\bx_\Gamma, \big(y_1,y_2,-\tfrac{2}{\sqrt{\delta}}\big)\big)
\end{equation*}
and using the second line of
relation~(\ref{eq:one_dimensional_interface_conditions}) gives
\begin{equation}
\label{eq:limit_bPsi_-}
\lim_{\delta \to 0} \int_{\Gamma^\delta_-(\sqrt{\delta})}
\bPsi^\delta_-(\bx_\Gamma, \by) \cdot \bn^\delta_-(\by) \,
\mathrm{d}\sigma(\by) = a_C V_0(\bx_\Gamma) \sin \big( \tfrac{\omega L}{c}
\big). 
\end{equation}

\subsection{Weak convergence in the apertures}
\label{sec:weak-conv-apert}

For each $\delta > 0$, we consider $\bx \in \Omega^\delta$ and corresponding
resonator position $\bx^\delta_\Gamma \in \Gamma^\delta$ such that one of the
three following conditions is satisfied:
\begin{enumerate}
    \item $\bx$ belongs to the neck $\Omega^\delta_N(\bx^\delta_\Gamma)$,
    \item $\bx$ is inside the resonator chamber
    $\Omega^\delta_C(\bx^\delta_\Gamma)$, and
    $\big| \bx - (\bx_\Gamma^\delta,-\delta^2 h_0) \big| < 2
    \delta^{\tfrac{3}{2}}$,
    \item $\bx$ is inside the domain $\Omega$, and
    $\big| \bx - (\bx_\Gamma^\delta,0) \big| < 2
    \delta^{\tfrac{3}{2}}$,
\end{enumerate}
\ie the distance from $\bx$ to the neck
$\Omega^\delta_N(\bx^\delta_\Gamma)$, above and below the aperture, is
at most $2 \delta^{\tfrac{3}{2}}$. We introduce then the domain $\Omega^\delta_A(\bx^\delta_\Gamma)$ as
the union of the neck $\Omega^\delta_N(\bx^\delta_\Gamma)$ and the two
half-spheres of diameter $2 \delta^{\tfrac{3}{2}}$, see
Fig.~\ref{fig:NNF}(a), and we introduce the variable change
$\bx = \bx^\delta_\Gamma + \delta^2 \bz$, \ie it is equivalent to
introduce the variable change $\by = \delta \bz$ in
Section~\ref{sec:weak-conv-patt}. As $\bx$ describes
$\Omega^\delta_A(\bx^\delta_\Gamma)$, $\bz$ describes the domain
$\widehat{\Omega}(2\delta^{-1/2})$ that tends to the unbounded domain
$\widehat{\Omega}$ as $\delta$ tends to $0$, see
Fig.~\ref{fig:NNF}(b) and Fig.~\ref{fig:NNF}(c). We also introduce the
five-dimensional functions $\bfv^\delta$ and $\fp^\delta$ by
\begin{equation}
\label{eq:change_function_aperture}
\bfv^\delta(\bx^\delta_\Gamma,\bz) = \delta^2 \bv^\delta(\bx^\delta_\Gamma +
\delta^2 \bz) \quad \text{and} \quad
\fp^\delta(\bx^\delta_\Gamma,\bz) = p^\delta(\bx^\delta_\Gamma +
\delta^2 \bz).
\end{equation}

The scale change for the velocity $\bfv^\delta$ is due to the following: for $\by \in \Gamma^\delta_-(\sqrt{\delta})$,
\ie for $\bz \in \widehat{\Gamma}_-(1/\sqrt{\delta})$, using the variable change $\by = \delta \bz$ leads to
\begin{equation*}
\int_{\Gamma^\delta_-(\sqrt{\delta})} \bv^\delta(\bx^\delta_\Gamma + \delta \by) \cdot \bn^\delta_-(\by) \,
\mathrm{d} \sigma(\by) = \delta^2 \int_{\widehat{\Gamma}_-(1/\sqrt{\delta})}
\bv^\delta(\bx^\delta_\Gamma + \delta \bz) \cdot \bn^\delta_-(\delta \bz) \, \mathrm{d} \sigma(\bz),
\end{equation*}
that leads to
\begin{equation}
\label{eq:matching_velocity_minus_aperture}
\int_{\Gamma^\delta_-(\sqrt{\delta})} \bPsi^\delta(\bx^\delta_\Gamma,\by) \cdot \bn^\delta_-(\by) \,
\mathrm{d} \sigma(\by) = \int_{\widehat{\Gamma}_-(1/\sqrt{\delta})} \bfv^\delta(\bx^\delta_\Gamma,\bz)
\cdot \bn_-(\bz) \, \mathrm{d} \sigma(\bz),
\end{equation}
with
\begin{equation*}
\bn_-(\bz) = \frac{(z_1,z_2,-(z_3+h_0))}{|(z_1,z_2,-(z_3+h_0))|}.
\end{equation*}

We consider the linearized Navier-Stokes
problem~\eqref{eq:Navier_Stokes} and we apply the isotropic coordinate
change. Similarly to the derivation of
problem~\eqref{eq:Navier_Stokes:resonator}, we obtain the following
system
\begin{subequations}
    \label{eq:Navier_Stokes:aperture}
    \begin{align}
    - \imath \omega \bfv^\delta +
    \tfrac{1}{\rho_0} \nabla_{\bz} 
    \fp^\delta -
    \nu_0 \Laplace_{\bz} \bfv^\delta - \nu'_0
    \nabla_{\bz} \Div_\bz \bfv^\delta
    & = \zerobf, \quad \text{in }
    \Gamma^\delta \times \widehat{\Omega}(2\delta^{-1/2}),
    \label{eq:Navier_Stokes:aperture:M}
    \\
    - \imath \omega \delta^4 \fp^\delta +
    \rho_0 c^2 \Div_{\bz}
    \bfv^\delta
    &= 0, \quad \text{in } \Gamma^\delta \times \widehat{\Omega}(2\delta^{-1/2}),
    \label{eq:Navier_Stokes:aperture:C}
    \\
    \bfv^\delta
    &= \zerobf , \quad \text{on }
    \Gamma^\delta \times \widehat{\Gamma}^\delta_A,
    \label{eq:Navier_Stokes:aperture:B}
    \end{align}
\end{subequations}
where
$\widehat{\Gamma}^\delta_A := \partial \widehat{\Omega}(2\delta^{-1/2}) \cap
\tfrac{1}{\delta^2} \big(\partial \Omega^\delta - \bx^\delta_\Gamma
\big)$ corresponds to the rescaled part of the boundary of
$\Omega^\delta$ in the vicinity of the neck centered at
$\bx^\delta_\Gamma$ (depicted in blue on Fig.~\ref{fig:NNF}(b)), and
tends to $\partial \widehat{\Omega}$ as $\delta$ tends to $0$.

Again, in view of estimate~\eqref{eq:lema_better_estimate} of
Lemma~\ref{lema:better_estimate}, the rescaled functions
$(\bfv^\delta,\fp^\delta)$ satisfy the following estimate for almost
all resonators
\begin{multline}
\label{eq:lemma_better_estimate_scaled_aperture}
\xnorm{\bfv^\delta
    (\bx^\delta_\Gamma,\cdot)}{\Ltwo(\widehat{\Omega}(2\delta^{-1/2}))}^2 +
\tfrac{1}{\delta^4}
\xnorm{\Div_{\bz}
    \bfv^\delta(\bx^\delta_\Gamma,\cdot)}{\Ltwo(\widehat{\Omega}(2\delta^{-1/2}))}^2
+
\xnorm{\Curl_{\bz}
    \bfv^\delta(\bx^\delta_\Gamma,\cdot)}{\Ltwo(\widehat{\Omega}(2\delta^{-1/2}))}^2
\\ + \delta^4 \xnorm{\fp^\delta
    (\bx^\delta_\Gamma,\cdot)}{\Ltwo(\widehat{\Omega}(2\delta^{-1/2}))}^2
+ 
\xnorm{\nabla \fp^\delta
    (\bx^\delta_\Gamma,\cdot)}{\Ltwo(\widehat{\Omega}(2\delta^{-1/2}))}^2
\leqslant C_H
\end{multline}
We extend then the discrete problem~\eqref{eq:Navier_Stokes:aperture}
into a continuous problem on the resonator position
$\bx_\Gamma \in \Gamma$ and the error
estimate~\eqref{eq:lemma_better_estimate_scaled_aperture} into an
$\Ltwo(\Gamma; \Ltwo(\widehat{\Omega}(2\delta^{-1/2})))$ estimate.

The sequence $\bfv^\delta$ is bounded in
$\Ltwo(\Gamma; \Hs(\Div,D) \cap \Hs(\Curl,D) )$ for any bounded open
set $D$ included in $\widehat{\Omega}$ and for any $\delta$ such
that $D \subset \widehat{\Omega}(2\delta^{-1/2})$, then we can extract a
subsequence (still denoted by $\bfv^\delta$) that converges weakly in
$\Ltwo(\Gamma; \Hs(\Div,D) \cap \Hs(\Curl,D) )$. Taking iteratively
$D = D^n := \widehat{\Omega}_{A} \cap \mB(\zerobf, 2^n)$, we state
that $\bfv^\delta$ admits a subsequence that converges weakly
to a function $\bfv_{-2}$ in
$\Ltwo(\Gamma ; \Hs(\Div, \widehat{\Omega}) \cap \Hs(\Curl,
\widehat{\Omega}))$, the subscript ``$-2$'' relates to the shift in
$\delta$ for the function $\bfv^\delta$. We deduce moreover from the
first line of~\eqref{eq:lemma_better_estimate_scaled_aperture} and
using the lower semi-continuity that $\Div \bfv_{-2} = 0$ in
$\widehat{\Omega}$.

Similarly, the sequence $\fp^\delta$ is bounded in
$\Ltwo(\Gamma; \bmV(D))$ for any bounded open set $D$ included in
$\widehat{\Omega}$ and for any $\delta$ such that
$D \subset \widehat{\Omega}(2\delta^{-1/2})$, where
\begin{equation*}
\bmV(D) = \Big\lbrace \fp \in \Ltwoloc(D) \text{ such that } \nabla \fp
\in \Ltwo(D)^3 \Big\rbrace,
\end{equation*}
then doing a similar construction, this sequence admits a subsequence
that converges weakly to $\nabla_\bz \fp_0$ in
$\bmV(\widehat{\Omega})$. Using then the weak convergence in
the continuity equation~\eqref{eq:Navier_Stokes:aperture:M}, we get
that the weak limit $(\bfv_{-2},\fp_0)$ is solution of the following
instationary Stokes problem
\begin{subequations}
    \label{eq:NNF}
    \begin{align}
    \label{eq:NNF:1}
    - \imath \omega {\bfv_{-2}}(\bx_\Gamma,\cdot) +
    \tfrac{1}{\rho_0}\nabla_{\bz} {\fp_0}(\bx_\Gamma,\cdot) -
    \nu_0 \Laplace_{\bz} {\bfv_{-2}}(\bx_\Gamma,\cdot)
    & = \zerobf, && \quad \text{in } \widehat{\Omega}, \\
    \label{eq:NNF:2}
    \Div_{\bz} {\bfv_{-2}}(\bx_\Gamma,\cdot)
    & = 0, && \quad \text{in } \widehat{\Omega}, \\
    {\bfv_{-2}}(\bx_\Gamma,\cdot)
    & = \zerobf , && \quad \text{on } \partial \widehat{\Omega}.
    \label{eq:NNF:3}
    \end{align}
\end{subequations}
Note that this problem can be written equivalently with a
$\Curl_{\bz} \Curl_{\bz}$ operator instead of the $- \Laplace_{\bz}$
operator, since $\Div_{\bz} {\bfv_{-2}}(\bx_\Gamma,\cdot) = 0$. Taking
the divergence gives that $p_0(\bx_\Gamma,\cdot)$ is an harmonic
function on $\widehat{\Omega}$, so that its behaviour towards
infinity is described using spherical functions\cite{hobson:1955}. Then,
following an expansion of $\bfv_{-2}$ as a sum of spherical functions
and functions that are exponentially decaying with respect to the
distance to the boundary $\partial \widehat{\Omega}$, the only
spherical harmonics on a half-sphere that lead to
$\bfv_{-2}(\bx_\Gamma,\cdot) \in \Hs(\Div,\widehat{\Omega}) \cap
\Hs(\Curl,\widehat{\Omega})$ are the spherical harmonics that
admit a behaviour at most constant towards infinity, the constants at
both sides of the wall can be different. Therefore, we seek for two
functions $c_\fm$ and $c_\fj$ defined on $\Gamma$ such that
\begin{equation}
\label{eq:definition_functions_aperture}
\big( \bfv_{-2}(\bx_\Gamma,\bz), \fp_0(\bx_\Gamma,\bz)
\big) = c_\fm(\bx_\Gamma) (\zerobf,1) + c_\fj(\bx_\Gamma)
\big( \bfv(\bz), \fp(\bz) \big),
\end{equation}
where the \emph{neck profile} $(\bfv,\fp) \in \big(
\Hs(\Div,\widehat{\Omega}) \cap \Hs(\Curl,\widehat{\Omega}) \big)
\times \bmV(\widehat{\Omega})$ is solution of the instationary Stokes
problem
\begin{subequations}
    \label{eq:canonical_NNF}
    \begin{align}
    \label{eq:canonical_NNF:1}
    - \imath \omega \bfv + \tfrac{1}{\rho_0}\nabla_\bz
    \fp - \nu_0 \Laplace_\bz 
    \bfv
    & = \zerobf, && \quad \text{in } \widehat{\Omega}, \\
    \label{eq:canonical_NNF:2}
    \Div_\bz \bfv & = 0, && \quad \text{in } \widehat{\Omega}, \\
    \bfv
    & = \zerobf , && \quad \text{on } \partial \widehat{\Omega},
    \label{eq:canonical_NNF:3}
    \intertext{completed by Dirichlet jump conditions at infinity}
    \label{eq:canonical_NNF_infinity}
    \lim_{S \to \infty}
    \fp_{|\widehat{\Gamma}_\pm(S)} &= \pm
    \tfrac{1}{2}, 
    \end{align}
\end{subequations}
where the half-spheres $\widehat{\Gamma}_\pm(S)$ for $S > 0.5 d_0$ are
given by
\begin{multline}
\label{eq:canonical_boundaries}
\widehat{\Gamma}_\pm(S) = \big\lbrace \bz \in
\widehat{\Omega} \text{ such that } \big| \bz
- (\pm 0.5 - 0.5 )h_0 \be_3 \big| = S \\ \text{ and } \pm ( \bz \cdot \be_3
\mp 0.5 h_0
+ 0.5 h_0 ) > 0 \big\rbrace.
\end{multline}
and are depicted on Fig.~\ref{fig:NNF}(c). 

Let us take $S > 0.5 d_0$. We denote by $\widehat{\Omega}(S)$ the
subdomain of $\widehat{\Omega}$ that is delimited by
$\widehat{\Gamma}_-(S)$ and $\widehat{\Gamma}_+(S)$. Since the function
$\bfv$ is divergence-free in $\widehat{\Omega}(S)$ and its
trace vanishes on $\partial \widehat{\Omega}$, it turns out
immediately that
\begin{equation*}
\int_{\widehat{\Gamma}_+(S)}
\tilde{\fv} \cdot \bn + \int_{\widehat{\Gamma}_-(S)}
\tilde{\fv} \cdot \bn = 0, 
\end{equation*}
where $\bn$ is the unit outward normal vector. Following the
formulation of the Rayleigh conductivity
$K_R$\cite{Rayleigh:1870,Rayleigh:1945} which describes the ratio of
the fluctuating volume flow to the driving pressure difference, we
introduce the \emph{effective Rayleigh conductivity} $k_R$ as
\begin{equation}
\label{eq:KR_CC}
k_R := \lim_{S \to \infty} \frac{\imath \omega
    \rho_0}{2} \Big( \int_{\widehat{\Gamma}_+(S)} 
\tilde{\fv} \cdot \bn - \int_{\widehat{\Gamma}_-(S)}
\tilde{\fv} \cdot \bn \Big).
\end{equation}

Existence and uniqueness of problem~\eqref{eq:canonical_NNF} is stated by Proposition~\ref{prop:well_posedness_aperture_problem},
and properties of the effective Rayleigh coefficient $k_R$ is stated by Proposition~\ref{propostion:effective_behaviour_kR}.

It remains to determine the conditions at infinity satisfied by $(\bfv,\fp)$,
using that for almost any $\by \in \Gamma^\delta_-(\sqrt{\delta})$,
\begin{equation*}
\begin{aligned}
\bPsi^\delta_-(\bx_\Gamma,\by) & = \bv^\delta(\bx_\Gamma + \delta \by) =
\bfv^\delta\big(\bx_\Gamma,\tfrac{\by}{\delta}\big), \\
\Phi^\delta_-(\bx_\Gamma,\by) & = p^\delta(\bx_\Gamma + \delta \by) = 
\fp^\delta\big(\bx_\Gamma,\tfrac{\by}{\delta}\big).
\end{aligned}
\end{equation*}

\textbf{For the rescaled velocity $\bfv^\delta$}, we use
relations~(\ref{eq:limit_bPsi_-}), (\ref{eq:matching_velocity_minus_aperture}) and~(\ref{eq:definition_functions_aperture}),
coupled with the definition of the effective Rayleigh conductivity $k_R$, to obtain
\begin{equation*}
a_C V_0(\bx_\Gamma) \sin \big( \tfrac{\omega L}{c} \big) =
\tfrac{k_R}{\imath \rho_0 \omega} c_\fj(\bx_\Gamma).
\end{equation*}
Using the solution representation~(\ref{eq:definition_functions_aperture}), 
we obtain the following relation
\begin{equation}
\label{eq:relationship_velocity_aperture_plus}
\lim_{\delta \to 0} \int_{\widehat{\Gamma}_+(1/\sqrt{\delta})} \bfv^\delta(\bx_\Gamma,\bz)
\cdot \bn_+(\bz) \, \mathrm{d} \sigma(\bz) = c_\fj(\bx_\Gamma) = \tfrac{\imath \rho_0 \omega
    a_C}{k_R} V_0(\bx_\Gamma) \sin \big( \tfrac{\omega L}{c} \big),
\end{equation}
where
\begin{equation*}
\bn_+(\bz) = \frac{(-z_1-z_2,z_3)}{|(-z_1,-z_2,z_3)|}.
\end{equation*}

\textbf{For the rescaled pressure $\fp^{\delta}$}, we use
relations~(\ref{eq:limit_Phi_-}) and~(\ref{eq:definition_functions_aperture})
to obtain
\begin{equation*}
- \imath \rho_0 c V_0(\bx_\Gamma) \cos \big( \tfrac{\omega L}{c} \big) =
c_\fm(\bx_\Gamma) - 0.5 c_\fj(\bx_\Gamma).
\end{equation*}
Using the again the solution representation~(\ref{eq:definition_functions_aperture}),
we obtain the following relation
\begin{equation}
\label{eq:relationship_pressure_aperture_plus}
\lim_{\delta \to 0}
\fp^\delta(\bx_\Gamma,\cdot)_{|\widehat{\Gamma}_+(1/\sqrt{\delta})} = -
\imath \rho_0 c V_0(\bx_\Gamma) \cos \big( \tfrac{\omega L}{c} \big)
+ \tfrac{\imath \rho_0 \omega
    a_C}{k_R} V_0(\bx_\Gamma) \sin \big( \tfrac{\omega L}{c} \big).
\end{equation}

\subsection{Weak convergence in the pattern above aperture}
\label{sec:weak-conv-patt-2}

Similarly to the formal derivation in the pattern below aperture in
Section~\ref{sec:weak-conv-patt}, for each $\delta > 0$ we consider $\bx \in
\Omega$ and and corresponding resonator position $\bx^\delta_\Gamma$ such that
$\bx \in \Omega^\delta_H(\bx^\delta_\Gamma)$. Then we
define the zone above aperture
\begin{equation*}
\Omega^\delta_+(\bx^\delta_\Gamma) := \left\lbrace \bx \in \bx^\delta_\Gamma
+ \delta \mA \times (0,2\sqrt{\delta}) \text{ such that } \bx \not\in
B_3(\bx^\delta_\Gamma, \delta^{\tfrac{3}{2}}) \right\rbrace, 
\end{equation*}
and the rescaled zone below aperture (see Fig.~\ref{fig:domain_aperture_plus})
for the variable change $\by = \delta^{-1} (\bx-\bx^\delta_\Gamma)$.

Here, we are interested in the behaviour of the solution $\bv^\delta,p^\delta)$
solution of~(\ref{eq:Navier_Stokes}) in the domain $\Omega$, close to the
apertures. To do so, we introduce the five-dimensional functions
$\bPsi^\delta_+$ and $\Phi^\delta_+$ by
\begin{equation}
\label{eq:change_function_aperture_above}
\bPsi^\delta_+(\bx^\delta_\Gamma, \by) = \bv^\delta(\bx^\delta_\Gamma
+ \delta \by) \quad \text{and} \quad \Phi^\delta_+(\bx^\delta_\Gamma,
\by) = \tfrac{1}{\delta} p^\delta(\bx^\delta_\Gamma + \delta \by),
\end{equation}
similarly to the introduction of the functions $\bPsi^\delta_-$ and
$\Phi^\delta_-$ in~\eqref{eq:change_function_aperture_below}.

\begin{figure}[!bt]
    \centering
    \begin{tikzpicture}[scale=0.6]
    \draw[white] (-2.5,-0.5) rectangle (2.5,0.5);
    \fill[black!20!white] (-3.5,0.345) rectangle (3.5,7);
    \filldraw[draw=red, fill=white] (1,0.345) arc(0:180:1) node
    [pos=0.5,above] {$\color{red} \Gamma^\delta_+(\sqrt{\delta})$};
    \draw[thick] (-1,0.345) -- (-0.346,0.345) -- (-0.346,-0.345) --
    (-3,-0.345);
    \draw[dashed] (-3.5,0.345) -- (-3.5,8);
    \draw[thick] (1,0.345) -- (0.346,0.345) -- (0.346,-0.345) --
    (3,-0.345);
    \draw[dashed] (3.5,0.345) -- (3.5,8);
    \draw[blue,thick] (-3.5,0.345) -- (-1,0.345);
    \draw[blue,thick] (1,0.345) -- (3.5,0.345) node [pos=1.0,right] {$\mG^\delta_+$};
    \draw[->] (0,0.345) ++ (0,0) -- ++ (60:1) node [pos=1.1,right]
    {$\sqrt{\delta}$}; 
    \draw[<->] (1.8,-0.345) -- (1.8,0.345) node[pos=0.5, right]
    {$h_0 \delta$};
    \draw[<->] (-0.346,-0.28) -- (0.346,-0.28) node[pos=0.5, below]
    {$d_0 \delta$};
    \draw (-3.5,7) -- (3.5,7) node[pos=0.5,below] {$y_3 = 2 / \sqrt{\delta}$};
    \filldraw (0,0.345) circle (0.1) node [above] {$\zerobf$};
    \end{tikzpicture}
    \caption{Representation of the canonical domain $\mB^\delta_+$ for the pattern
        above aperture (gray).}
    \label{fig:domain_aperture_plus}
\end{figure}
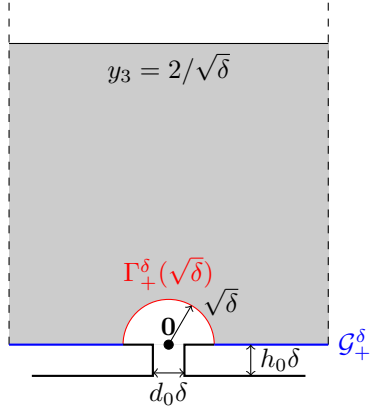

Then, in view of estimate~(\ref{eq:lema_better_estimate}) of
Lemma~\ref{lema:better_estimate} the rescaled functions $(\bPsi^\delta_-,
\Phi^\delta_-)$ satisfy the following estimate for almost all resonator
\begin{multline}
\label{eq:lemma_better_estimate_scaled_pattern+}
\delta
\xnorm{\bPsi^\delta_+(\bx^\delta_\Gamma,\cdot)}{\Ltwo(\mB^\delta_+)}^2
+ \tfrac{1}{\delta} \xnorm{\Div_{\by}
    \bPsi^\delta_+(\bx^\delta_\Gamma,\cdot)}{\Ltwo(\mB^\delta_+)}^2
+ \delta^3 \xnorm{\Curl_{\by}
    \bPsi^\delta_+(\bx^\delta_\Gamma,\cdot)}{\Ltwo(\mB^\delta_+)}^2
\\ + \delta
\xnorm{\Phi^\delta_+(\bx^\delta_\Gamma,\cdot)}{\Ltwo(\mB^\delta_+)}^2
+ \tfrac{1}{\delta} \xnorm{\nabla_\by
    \Phi^\delta_+(\bx^\delta_\Gamma,\cdot)}{\Ltwo(\mB^\delta_+)}^2
\leqslant C_H.
\end{multline}
We extend again this \emph{discrete} error estimate into a \emph{continuous}
error estimate on $\Ltwo(\Gamma,\Ltwo(\mB^\delta_+))$. This error estimate
allows us to give some results about the rescaled functions.

\textbf{For the rescaled pressure $\Phi_+^\delta$}, using the Cauchy-Schwartz
inequality and the error
estimate~(\ref{eq:lemma_better_estimate_scaled_pattern+}) leads to
\begin{equation*}
\Big( \int_{\mB^\delta_+} |\nabla_\by \Phi^\delta_+(\bx_\Gamma,\by)| \,
\mathrm{d}\by \Big)^2 \leqslant \frac{2}{\sqrt{\delta}} \int_{\mB^\delta_+}
|\nabla_\by \Phi^\delta_+(\bx_\Gamma,\by)|^2 \,
\mathrm{d}\by \leqslant 2 C_H \sqrt{\delta},
\end{equation*}
and this quantity tends to $0$ as $\delta \to 0$, so that the gradient of
$\Phi^\delta_+(\bx_\Gamma,\cdot)$ tends to $0$ almost everywhere in $\mA \times
(0, \infty)$. It remains to check that the average of
$\Phi^\delta_+$ remains uniformly bounded as well. This can be checked using
again the energy estimate~(\ref{eq:lemma_better_estimate_scaled_pattern+}) and
using that
\begin{equation*}
\delta \xnorm{1}{\Ltwo(\mB^\delta_+)}^2 = \delta \Big( 
\tfrac{2}{\sqrt{\delta}}-
\tfrac{2\pi}{3}\delta^{\tfrac{3}{2}} \Big)^2.
\end{equation*}
We can then extract a subsequence (that we still denote by $\Phi^\delta_-$)
that converges to a constant. Using
\begin{equation*}
\fp^\delta(\bx_\Gamma, \delta^{-1} \by) = p^\delta(\bx_\Gamma + \delta \by) = \Phi^\delta_-(\bx_\Gamma,\by)
\end{equation*}
for almost every $\by \in \Sigma^\delta_+(\sqrt{\delta})$
and~(\ref{eq:relationship_pressure_aperture_plus}), we obtain that
\begin{equation}
\label{eq:limit_Phi_+}
\lim_{\delta \to 0} \Phi^\delta_+ = - \imath \rho_0 c V_0(\bx_\Gamma) \cos
\big( \tfrac{\omega L}{c} \big) + \tfrac{\imath \rho_0 \omega a_C}{k_R}
V_0(\bx_\Gamma) \sin \big( \tfrac{\omega L}{c} \big).
\end{equation}

\textbf{Similarly, for the rescaled velocity $\bPsi^\delta_+$}, using the
Cauchy-Scwhartz inequality leads to
\begin{equation*}
\Big( \int_{\mB^\delta_+} |\Div_\by \bPsi^\delta_+(\bx_\Gamma,\by)| \,
\mathrm{d}\by \Big)^2 \leqslant \frac{2}{\sqrt{\delta}} \int_{\mB^\delta_+}
|\Div_\by \bPsi^\delta_+(\bx_\Gamma,\by)|^2 \,
\mathrm{d}\by \leqslant 2 C_H \sqrt{\delta},
\end{equation*}
and this quantity tends to $0$ as $\delta$ tends to $0$. Using the Gauss
theorem on $\mB^\delta_+$ gives
\begin{multline*}
\int_{\mB^\delta_+} \Div_\by \bPsi^\delta_+= \int_{\mA} \bPsi^\delta_+
\big(\bx_\Gamma,\big(y_1,y_2,\tfrac{2}{\sqrt{\delta}}\big)\big) \cdot \be_3 \,
\mathrm{d}y_1 \, \mathrm{d}y_2 \\ - \int_{\Gamma^\delta_+(\sqrt{\delta})}
\bPsi^\delta_+(\bx_\Gamma, \by) \cdot \bn_+(\by) \, \mathrm{d}\sigma(\by) +
\int_0^{2/\sqrt{\delta}} \int_{\partial \mA_C} \bPsi_+(\bx_\Gamma,\by) \cdot
\bn \, \mathrm{d} \sigma(y_1,y_2) \, \mathrm{d} y_3 .
\end{multline*}
Since $\mA$ is a parallelogram driven by the two vectors $(\ba_i)_{i \in
    \lbrace 1,2 \rbrace}$, we call $\Gamma_{\mA,i}$ the edge of $\mA$ such that
$\Gamma_{\mA,i} + \ba_i$ is also one edge of $\mA$. Then, for $\bx_\Gamma \in \Gamma$
and $\by \in
\Gamma_{\mA_i} \times (0,\infty)$, we consider the point $\bx := \bx_\Gamma +
\delta \by + \delta \ba_i$ that corresponds locally to a common boundary of the
two semi-infinite strips centered respectively at $\bx_\Gamma$ and $\bx_\Gamma
+ \delta \ba_i$:
\begin{equation}
\label{eq:Psi_+_boundaries}
\bPsi^\delta_+(\bx_\Gamma,\by + \ba_i) = \bv^\delta(\bx_\Gamma + \delta \by +
\delta \ba_i) = \bPsi^\delta_+(\bx_\Gamma + \delta \ba_i,\by).
\end{equation}
Since both the norms of $\Div \bPsi^\delta_+$ and $\Curl \bPsi^\delta_+$ are
decaying to $0$ using~(\ref{eq:lemma_better_estimate_scaled_pattern+}), the
norm of the gradient of each component of $\bPsi^\delta_+$ also tends to $0$ as
$\delta$ tends to $0$. Using 
\begin{equation*}
\bPsi^\delta_+(\bx_\Gamma,\by + \ba_i) = \bPsi^\delta_+(\bx_\Gamma,\by) +
\int_{0}^1 \nabla \bPsi^\delta_+(\bx_\Gamma,\by + t \ba_i) \cdot \ba_i \, \mathrm{d}t,
\end{equation*}
we deduce that
\begin{equation*}
\lim_{\delta \to 0}  \int_0^{2/\sqrt{\delta}} \int_{\partial \mA_C}
\bPsi_+(\bx_\Gamma,\by) \cdot \bn \, \mathrm{d} \sigma(y_1,y_2) \, \mathrm{d}
y_3 = 0.
\end{equation*}
Using finally~(\ref{eq:relationship_velocity_aperture_plus}), we obtain
\begin{equation}
\label{eq:limit_bPsi_+}
\lim_{\delta \to 0} \int_{\mA} \bPsi^\delta_+
\big(\bx_\Gamma,\big(y_1,y_2,\tfrac{2}{\sqrt{\delta}}\big)\big) \cdot \be_3 \,
\mathrm{d}y_1 \, \mathrm{d}y_2 = a_C V_0(\bx_\Gamma) \sin
\big( \tfrac{\omega L}{c} \big).
\end{equation}

\subsection{Weak convergence in the macroscopic region}
Now, we are interested in the behaviour of the solution $(\bv^\delta,p^\delta)$
of~(\ref{eq:Navier_Stokes}). Using the assumption estimate~(\ref{eq:universal_estimate}),
there exists a subsequence that weakly converges to a limit $(\bv_0,p_0)$ in $\Hs(\Div,\Omega) \times \Hone(\Omega)$.

The weak convergence applied to the continuity
equation~(\ref{eq:Navier_Stokes:C}) gives immediately the second line
of~(\ref{eq:limit_term_complete}). Multiplying the momentum
equation~(\ref{eq:Navier_Stokes:M}) by a test function $\bw \in \Hs(\Div;
\Omega) \cap \Hs(\Curl; \Omega)$ such that $\bw = 0$ on $\partial \Omega$ and
using the Gauss theorem leads to
\begin{multline*}
-\imath \omega \xscal{\bv^\delta}{\bw}{\Omega} + \tfrac{1}{\rho_0}
\xscal{\nabla \rho^\delta}{\bw}{\Omega} + \nu_0 \delta^4
\xscal{\Curl \bv^\delta}{\Curl \bw}{\Omega} \\ + (\nu_0 + \nu_0')
\delta^4 \xscal{\Div \bv^\delta}{\Div \bw}{\Omega} =
\xscal{\bff}{\bw}{\Omega}.
\end{multline*}
Using then the weak convergence associated to the boundness of the norms
$\delta^2 \xnorm{\Curl \bv^\delta}{\Ltwo(\Omega)}$ and
$\xnorm{\Div \bv^\delta}{\Ltwo(\Omega)}$ leads to the first line
of~\eqref{eq:limit_term_complete}.

Next point is to derive the boundary condition. The easiest part is to derive
the boundary condition on $\partial \Omega \setminus \Gamma$. Indeed, the trace
operator
\begin{align*}
\gamma_0 : \quad \Hs(\Div; \Omega)
&\to \Ltwo(\partial \Omega \setminus
\Gamma),
\\
\bv & \mapsto \bv \cdot \bn,
\end{align*}
is a lower semi-continuous operator, and since $\bv^\delta$ weakly converges to
$\bv_0$ in $\Hs(\Div; \Omega)$, one has
\begin{equation*}
\xnorm{\bv_0 \cdot \bn}{\Ltwo(\partial \Omega \setminus \Gamma)}
\leqslant \liminf_{\delta \to 0}
\xnorm{\bv^\delta \cdot \bn}{\Ltwo(\partial \Omega \setminus
    \Gamma)} = 0.
\end{equation*}
Determination of the boundary condition on $\Gamma$ is more involved,
and need the matching with the solution in the pattern above
apertures. Indeed, for a particular resonator position $\bx_\Gamma \in \Gamma$,
we consider the domain $\mO^\delta_+ = \bx^\delta_\Gamma + \delta \mA \times
(\sqrt{\delta},2\sqrt{\delta}) \subset \Omega$. We can moreover see
that the point $\by := \delta^{-1} \big( \bx - \bx_\Gamma \big)$
belongs to $\mB^\delta_+$. Then, similarly to the writing of the
conditions~(\ref{eq:limit_Phi_-}) and~(\ref{eq:limit_bPsi_-}), we get the matching
\begin{equation}
\label{eq:matching_above}
\begin{aligned}
\int_{\mA}
p^\delta\big(\bx_\Gamma + \delta \big(y_1,y_2,\tfrac{2}{\sqrt{\delta}}\big)\big) \, \mathrm{d}(y_1,y_2) & =
\int_{\mA}
\Phi^\delta_+\big(\bx_\Gamma,\big(y_1,y_2,\tfrac{2}{\sqrt{\delta}}\big)\big)
\, \mathrm{d}(y_1,y_2), \\
\int_{\mA}
\bv^\delta\big(\bx_\Gamma + \delta \big(y_1,y_2,\tfrac{2}{\sqrt{\delta}}\big)\big) \cdot \be_3 \, \mathrm{d}(y_1,y_2) &
= \int_{\mA}
\bPsi^\delta_+\big(\bx_\Gamma,\big(y_1,y_2,\tfrac{2}{\sqrt{\delta}}\big)\big)
\be_3 \, \mathrm{d}(y_1,y_2).
\end{aligned}
\end{equation}
The right-hand sides of~\eqref{eq:matching_above} are treated
using~(\ref{eq:limit_Phi_+}) and~(\ref{eq:limit_bPsi_+}), respectively. The left-hand sides are
treated using the $\Ltwo$ weak convergence of $\bv^\delta \cdot \be_3$
and $p^\delta$ to $\bv_0 \cdot \be_3$ and $p_0$ respectively and using
an elliptic regularity result. Therefore, it holds
\begin{equation}
\label{eq:matching_above_limit_replaced}
\begin{aligned}
\bv_0(\bx_\Gamma) \cdot \be_3 & = a_C
\sin \big( \tfrac{\omega L}{c} \big) V_0(\bx_\Gamma), \\
p_0(\bx_\Gamma) & = \Big( - \imath \omega \rho_0 \cos
\big( \tfrac{\omega L}{c} \big) + \tfrac{\imath \omega a_C
    \rho_0}{k_R} \sin \big( \tfrac{\omega L}{c} \big) \Big)
V_0(\bx_\Gamma),
\end{aligned}
\end{equation}
where the function $V_0(\bx_\Gamma)$ is still unknown. It can still be
eliminated so that the first line of~(\ref{eq:limit_term_complete}) is also
proved as well.

\subsection{Uniqueness of the limit}
\label{sec:uniqueness_limit}

We recall here the statement of the Lemma~\ref{lema:well-posedness_limit} about
the existence and uniqueness of the limit
problem~(\ref{eq:limit_term_complete}): let $\bff \in \Hs(\Div; \Omega)$, then
there exists a unique solution $(\bv_0,p_0) \in \Hs(\Div; \Omega) \times
\Hone(\Omega)$, except for frequencies $\omega \in \Lambda$, where $\Lambda$ is
a subset of $\tfrac{\pi c}{L} \IN$.

\begin{proof}[Proof of Lemma~\ref{lema:well-posedness_limit}]
    Due to the equivalence of problems~\eqref{eq:limit_term_complete}
    and~\eqref{eq:limit_term_complete_pressure}, and since
    $\Div \bff \in \Ltwo(\Omega)$, we seek for a solution
    $p_0 \in \Hone(\Omega)$ of this problem. Multiplying the first line
    of~\eqref{eq:limit_term_complete_pressure} by a test function
    $q \in \Hone(\Omega)$ and integrating by parts, the variational
    formulation associated to this problem is: find
    $p_0 \in \Hone(\Omega)$ such that, for any $q \in \Hone(\Omega)$,
    \begin{equation}
    \label{eq:limit_term_pressure_vf}
    \xscal{p_0}{q}{\Hone(\Omega)} - \big( 1 + \tfrac{\omega^2}{c^2}
    \big) \xscal{p_0}{q}{\Ltwo(\Omega)} + b_\Gamma(p_0,q) = - \rho_0
    \xscal{\Div \bff}{q}{\Ltwo(\Omega)}, 
    \end{equation}
    where $\xscal{p_0}{q}{\Hone(\Omega)}$ (respectively
    $\xscal{p_0}{q}{\Ltwo(\Omega)}$) is the inner scalar product of
    $p_0$ and $q$ in $\Hone(\Omega)$ (resp. in $\Ltwo(\Omega)$), and
    $b_\Gamma(p_0,q)$ is the boundary operator given by
    \begin{equation}
    \label{eq:boundary_operator_b_gamma}
    b_\Gamma(p_0,q) := \frac{\sin \big(\tfrac{\omega L}{c}\big)}{\frac{c}{ \omega
            a_C} \cos \big(\tfrac{\omega L}{c}\big) - \frac{2}{k_R} \sin \frac{\omega
            L}{c}} \xscal{p_0}{q}{\Ltwo(\Gamma)}.
    \end{equation}
    Since $\Im(k_R) > 0$ and all other quantities of the expression
    $\tfrac{c}{ \omega a_C} \cos \big( \tfrac{\omega L}{c} \big)- \tfrac{2}{k_R}
    \sin \big( \tfrac{\omega L}{c} \big)$ are real-valued, that expression never vanishes. The subspace
    $\Hone(\Omega)$ is compactly embedded in $\Ltwo(\Omega)$ by the
    Rellich-Kondrachov theorem, \cf Chapter 2 of the book of Braess\cite{Braess:2007}. Similarly,
    the trace operator $\gamma_0 : u \mapsto u_{|\Gamma}$ is a
    continuous operator from $\Hone(\Omega)$ to $\Hhalf(\Gamma)$, and
    again the subspace $\Hhalf(\Gamma)$ is compactly embedded into
    $\Ltwo(\Gamma)$. Hence the left-hand side of the variational
    formulation~\eqref{eq:limit_term_pressure_vf} can be written under
    the form $\xscal{(I + K) p_0}{q}{\Hone(\Omega)}$, where the operator
    $I$ is the identity operator and the operator $K$ defined by
    \begin{equation*}
    \xscal{K p}{q}{\Hone(\Omega)} := - \big( 1 + \tfrac{\omega^2}{c^2}
    \big) \xscal{p}{q}{\Ltwo(\Omega)} + b_\Gamma(p,q), \quad \forall
    p,q \in \Hone(\Omega),
    \end{equation*}
    is compact. Hence, the sum $I+K$ is a Fredholm
    operator of index~$0$\cite{Sauter.Schwab:2011}, \ie the dimension
    of its kernel coincides with the co-dimension of its range, and by
    the Fredholm alternative uniqueness implies existence.
    
    Assume now that $\xscal{(I + K) p_0}{q}{\Hone(\Omega)} = 0$ for any
    test function $q \in \Hone(\Omega)$. From now on, we consider two
    different cases, depending on the nature of the operator $b_\Gamma$:
    \begin{enumerate}
        \item $\sin \big( \tfrac{\omega L}{c} \big)= 0$, \ie
        $\omega \in \tfrac{\pi c}{L} \IN$. In that case, the problem
        admits a unique solution when $\tfrac{\omega^2}{c^2}$ is not an
        eigenvalue of the $-\Delta$ operator in $\Omega$. We denote then
        by $\Lambda$ the subset of $\tfrac{\pi c}{L} \IN$ such that
        $\tfrac{\omega^2}{c^2}$ is also an eigenvalue of the $-\Delta$
        operator in $\Omega$.
        \item $\sin \big( \tfrac{\omega L}{c} \big) \not= 0$, \ie
        $\omega \not\in \tfrac{\pi c}{L} \IN$. In that case, taking the particular
        test function $q = \overline{p_0}$ and then the imaginary part of
        $\xscal{(I + K) p_0}{\overline{p_0}}{\Hone(\Omega)}$, we find that
        $\Im \, b_\Gamma(p_0,{\overline{p_0}}) = 0$. Hence $p_0$
        vanishes on $\Gamma$. Due to the boundary condition on~$\Gamma$, also
        $\nabla p_0 \cdot \bn$ vanishes on $\Gamma$. Using
        then the unique continuation theorem on elliptic
        operators\cite{Protter:1960} (see also Section~4.3 of the book of
        Leis\cite{LeisBook}), we deduce that $p_0$ vanishes in whole $\Omega$ and
        therefore existence and uniqueness of a solution $p_0$
        of~\eqref{eq:limit_term_complete_pressure} follow.
    \end{enumerate}
\end{proof}

Existence and uniqueness of Lemma~\ref{lema:well-posedness_limit} states that the whole sequence $(\bv_\delta,p_\delta)$ weakly
converges to $(\bv_0,p_0)$ is $\Hs(\Div,\Omega) \times \Hone(\Omega)$, and not only a subsequence.

\section{Numerical simulations}
\label{sec:numerical-results}

In this section we describe first how we compute numerically the effective
Rayleigh conductivity $k_R$ (Section~\ref{sec:numer-comp-effect}), study the
normalized specified acoustic impedance as a function of frequency
(Section~\ref{sec:study-acoust-imped}) and the frequency-dependent disspation
in a wave-guide predicted by the derived model
(Section~\ref{sec:study-dissipation}) in comparison with existing models in
the literature.

We consider an array of Helmholtz resonators in a cylindrical duct of fixed
geometric mean $\delta = \sqrt{\delta_1 \delta_2}=8.5\,\text{mm}$ of
longitudinal and azimuthal inter-hole distances $\delta_1$ and $\delta_2$ (the
model depends on $\delta$, not on $\delta_1$ and $\delta_2$ separately).  The
aperture of the Helmholtz resonator is a cylinder of diameter
$d_\delta = 1\,\text{mm}$ and height $h_\delta = 1\,\text{mm}$. The relative
area of each resonator chamber is $a_C = 0.9$ such that the area of cross
section is $a_C \delta^2 = 65.025\,\text{mm}^2$. We show results for different
resonator depths $L$.

For all computations we consider the following physical parameters. The mean
density of the air is $\rho_0 = 1.2252 \, \text{kg}\,\text{m}^{-3}$, the
speed of sound is $c=340.45 \,\text{m}\,\text{s}^{-1}$ and the
viscosity is $\nu = 14.66\times10^{-6}\,\text{m}^2 \text{s}^{-1}$, see Tables
A.3, A.4 and A.7 of Lahiri\cite{Lahiri:2014} for mean pressure
$p = 101.325\,\text{kPa}$ and temperature $T = 288.15\,\text{K}$.

\subsection{Numerical computation of the effective Rayleigh conductivity}
\label{sec:numer-comp-effect}

In this section, we describe how to compute numerically approximations of the
solution $(\bfv,\fp)$ of the characteristic
problem~(\ref{eq:canonical_problem}) around one hole and from this an
approximations to the effective Rayleigh conductivity. For this we consider the
truncated domains $\widehat{\Omega}(S)$ of radius $S$ where we have to impose
additional boundary conditions at the half spheres $\Gamma_\pm(S)$.  It can be
justified similarly to a previous study of a macroscopic problem with boundary
layer effects\cite{Delourme.Schmidt.Semin:2016} that acoustic velocity profile
functions satisfies the decay condition $\nabla \bfv^\top \bn \to \zerobf$ when
$|\bz| \to \infty$. Hence, we consider the \emph{truncated characteristic
    problem} with homogeneous Neumann boundary conditions for $\bfv_S$ on
$\Gamma_\pm(S)$: Seek
$(\bfv_S,\fp_S) \in \Hone(\widehat{\Omega}(S))^3 \times
\Ltwo(\widehat{\Omega}(S))$ solution of
\begin{equation}
\label{eq:canonical_problem_truncated}
\begin{aligned}
- \imath \omega \bfv_S + \tfrac{1}{\rho_0}\nabla
\fp_S - \nu_0 \Laplace
\bfv_S & = \zerobf, && \quad \text{in } \widehat{\Omega}(S), \\
\Div \bfv_S & = 0, && \quad \text{in } \widehat{\Omega}(S), \\
\bfv_S & = \zerobf , && \quad \text{on } \partial
\widehat{\Omega}(S) \cap \partial \widehat{\Omega}, \\
\nabla \bfv_S^\top \bn & = \zerobf, && \quad \text{on }\Gamma_\pm(S), \\
\fp_S &= \pm \tfrac{1}{2}, && \quad \text{on }\Gamma_\pm(S) .
\end{aligned}
\end{equation}
The truncated characteristic problem is well-posed which can be similarly shown
as the well-posedness of~(\ref{eq:canonical_problem}) (see proof of
Proposition~\ref{prop:well_posedness_aperture_problem}).

\begin{prop}
    There exists a unique solution
    $(\bfv_S,\fp_S) \in \Hone(\widehat{\Omega}(S))^3 \times
    \Ltwo(\widehat{\Omega}(S))$ of~(\ref{eq:canonical_problem_truncated}).
\end{prop}

Following the formulation of the effective Rayleigh conductivity $k_R$ defined
by~(\ref{eq:KR}), we introduce the approximate effective Rayleigh conductivity
$k_R(S)$ as
\begin{equation*}
k_R(S) :=  \frac{\imath \omega \rho_0}{2} \Big( \int_{\widehat{\Gamma}_+(S)}
\bfv_S \cdot \bn - \int_{\widehat{\Gamma}_-(S)}
\bfv_S \cdot \bn \Big).
\end{equation*}
We obtain an even more accurate approximation by computing $k_R(S)$ for
several values of $S$ and an extrapolation for $S \to \infty$ with a polynomial in $1/S$.
In the examples in this section we computed for $S = 40, 45, 50, 55, 60$.

\subsection{Study of the acoustic impedance and local resonance frequency}
\label{sec:study-acoust-imped}

Following the works of Webster in the 1910s\cite{Webster:1919}, the thesis of
C.~Lahiri\cite{Lahiri:2014} and the references within, the normalized specified
acoustic impedance $\zeta$ is defined by
\begin{equation}
\label{eq:acoustic_impedance}
\zeta := - \tfrac{\overline{p_0}}{c \rho_0 \overline{\bv_0} \cdot \bn}.
\end{equation}
Note that on our definition there is
a complex conjugate and a change of sign as it was stated in our previous
work\cite{Schmidt.Semin.ThoensZueva.Bake:2018}.

For the impedance boundary condition given by the third line
of~(\ref{eq:limit_term_complete}) and by identification, the normalized
specified acoustic impedance for array of Helmholtz resonators is given by
\begin{equation}
\label{eq:acoustic_impedance_value}
\zeta_{\textsc{AHM-3v}} := \tfrac{-\omega \Im(k_R)}{c |k_R|^2} + \imath \Big( \tfrac{\omega
    \Re(k_R)}{c |k_R|^2} - \tfrac{1}{a_C} \cot \big( \tfrac{\omega L}{c} \big) \Big),
\end{equation}
the subscript AHM-3v stands for the proposed \emph{three-scale asymptotic homogenization method with viscosity}.

We compare formula~(\ref{eq:acoustic_impedance_value}) with the most recent
analytic formula taking into account viscosity (in a stagnant flow) in the
acoustics literature, that is the formula of Guess\cite{Guess:1975} 
\begin{equation}
\label{eq:acoustic_impedance_Guess}
\zeta_{\textsc{Gue}} = (1+\imath) \tfrac{\sqrt{8\omega\nu}}{c\sigma} \big(1+
\tfrac{h_\delta}{d_\delta} \big) + \tfrac{\imath \omega (h_\delta +
    \delta_{\textsc{cor}})}{c\sigma} + \tfrac{\omega^2 d^2}{8 c^2 \sigma} - \tfrac{\imath}{1-\varepsilon} \cot \big(
\tfrac{\omega L}{c} \big)\ ,
\end{equation}
extended by a correction $1-\varepsilon$ for thickness of the side walls of the resonance chamber\cite{Kooi.Sarin:1981}.
In view of~\eqref{eq:acoustic_impedance_value} we choose $\varepsilon = a_C$.
Here, $\sigma := \pi d_\delta^2 / (4\delta^2)$
designates the porosity of the periodic array of apertures which takes value
$\sigma
= 0.011$ in our study. Moreover,
$\delta_{\textsc{cor}}$
stands for the so-called end-point corrector for which different formulas exist
in literature.  We will use in our study the end-point corrector
$\delta_{\textsc{cor}} = 8d_\delta / (3\pi)$
of Morse\cite{Morse:1948} that takes the value
$\delta_{\textsc{cor}} = 0.849\,\text{mm}$ and of Ingard\cite{Ingard:1953}
based of a series expansion that takes the value
$\delta_{\textsc{cor}} = 0.709\,\text{mm}$ 
(when respecting the interaction with chamber walls through the parameter $a_C$).
Note that Ingard's approximate formula 
$\delta_{\textsc{cor}} = 8d_\delta /(3\pi)\big(1-1.25\sqrt{\sigma/\pi}(1+1/\sqrt{a_C})\big)$ 
with two terms gives $\delta_{\textsc{cor}} = 0.720\,\text{mm}$ and
so almost the same value, and, hence, we use only Ingard's original formula.
The end-point correction $\delta_{\textsc{cor}} = 8d_\delta /(3\pi)(1-0.7\sqrt{\sigma})$ used by Guess\cite{Guess:1975} seems 
to neglect the interaction with the chamber walls at all and we do not take it into account in the study.


As for the formula of Guess the normalized specified acoustic impedance
$\zeta_{\textsc{AHM-3v}}$ of our approach is a combination of an impedance of
the aperture and a reactance of the Helmholtz resonator, where normalized
specified acoustic resistance and reactance denote the real and imaginary part
of the acoustic impedance $\zeta$, respectively.  Hence, the resistance
$\Re(\zeta_{\textsc{AHM-3v}})$ is -- as for Guess' model -- independent of the
resonator chamber depth~$L$.  For Guess' model it is even independent of the
choice of the end-point correction.  It admits a pole when
$\tfrac{\omega L}{c}$ is a multiple of $\pi$, \ie when $L$ corresponds to a
multiple of $\tfrac{\lambda}{2}$, when denoting by
$\lambda:=\tfrac{2\pi c}{\omega}$ the characteristic one-dimensional wavelength
in the Helmholtz resonator.  For the resonator depth $L=100\,\text{mm}$ in our
study the poles of the impedance are multiples of $1702\,\text{Hz}$. Moreover,
the resistance $\Re(\zeta_{\textsc{ckr-v}})$ is a positive quantity since
$\Im(k_R) < 0$ thanks to Proposition~\ref{propostion:effective_behaviour_kR},
which is also the case for Guess' model.


\begin{figure}[!bt]
    \centering
    \begin{tikzpicture}[scale=0.8]
    \begin{axis}[
    xlabel=Frequency in Hz,
    xmin=0, xmax=2000,
    ymin=-3, ymax=3,
    ylabel=Reactance $\Im(\zeta)$,
    width=7cm, height=8cm,
    line width=1pt,
    y label style={at={(axis description cs:-0.1,.5)},anchor=south},
    legend style={at={(0.1,1.1)},anchor=west,legend columns=-1}]
    \addplot[smooth,color=black]
    table [x=freq, y=imag]{zeta_AHM-3v_L100.txt};
    \addlegendentry {AHM-3v};
    \addplot[dashed,smooth,color=black]
    table [x=freq, y=imag]{zeta_guess+ingard_L100.txt};
    \addlegendentry {Guess w. Ingard's end-corr.};
    \addplot[dashdotted,smooth,color=black]
    table [x=freq, y=imag]{zeta_guess+morse_L100.txt};
    \addlegendentry {Guess w. Morse' end-corr.};
    \addplot[smooth,color=black]
    table [x=freq, y=imag]{zeta_AHM-3v_L100_2.txt};
    \addplot[dashed,smooth,color=black]
    table [x=freq, y=imag]{zeta_guess+ingard_L100_2.txt};
    \addplot[dashdotted,smooth,color=black]
    table [x=freq, y=imag]{zeta_guess+morse_L100_2.txt};
    \draw[dashed, thin, color=black] (axis cs:1702,-3) -- (axis cs:1702,3); 
    \draw[solid,  thin, color=black] (axis cs:50,0) -- (axis cs:2000,0);
    \end{axis}
    \begin{axis}[
    xlabel=Frequency in Hz,
    xshift=7cm,
    xmin=0, xmax=2000,
    ymin=0, ymax=1,
    width=7cm, height=8cm,
    ylabel=Resistance $\Re(\zeta)$,
    line width=1pt,
    y label style={at={(axis description cs:-0.1,.5)},anchor=south},
    legend style={at={(0.1,0.8)},anchor=west}]
    \addplot[smooth,color=black]
    table [x=freq, y=real]{zeta_AHM-3v_L100.txt};
    \addplot[smooth,color=black,dashdotted]
    table [x=freq, y=real]{zeta_guess+ingard_L100.txt};
    \addplot[smooth,color=black]
    table [x=freq, y=real]{zeta_AHM-3v_L100_2.txt};
    \addplot[smooth,color=black,dashdotted]
    table [x=freq, y=real]{zeta_guess+ingard_L100_2.txt};
    \draw[dashed, thin, color=black] (axis cs:1702,-3) -- (axis cs:1702,3); 
    \end{axis}
    \end{tikzpicture}
    \caption{Reactance $\Im(\zeta)$ ({\em left})
        and resistance $\Re(\zeta)$ ({\em right})
        in dependence of frequency
        $f = \tfrac{\omega}{2\pi}$ for
        $d_\delta=h_\delta=1\,\text{mm}$, $\delta=8.5\,\text{mm}$, $L = 100\,\text{mm}$ and $a_C=0.9$.
        Note that the resistance 
        is independent of $L$ for all models and
        for Guess' model independent of the end-correction.
        The vertical line correspond to $\lambda/2 = 1702\,\text{Hz}$ where the reactance $\Im(\zeta)$ admits a pole
        and the array of Helmholtz resonators acts as a sound-hard wall.}
    \label{fig:plot_impedance}
\end{figure}
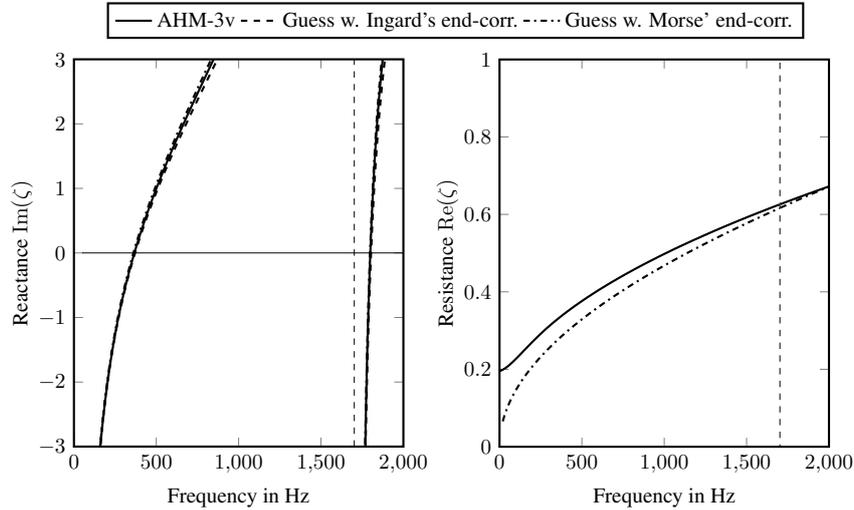

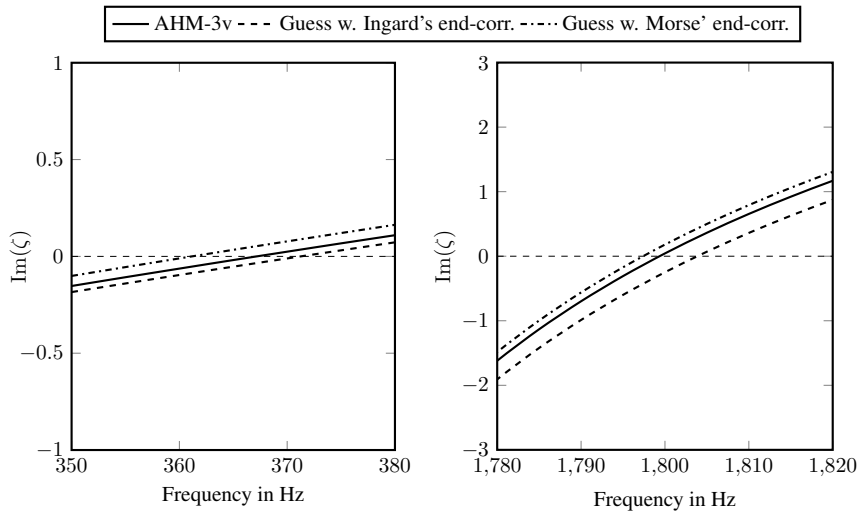
\begin{figure}[!bt]
    \centering
    \begin{tikzpicture}[scale=0.8]
    \begin{axis}[
    xlabel=Frequency in Hz,
    xmin=350, xmax=380,
    ymin=-1, ymax=1,
    ylabel=$\Im(\zeta)$,
    line width=1pt,
    width=6.9cm, height=8cm,
    y label style={at={(axis description cs:-0.1,.5)},anchor=south},
    legend style={at={(0.1,1.1)},anchor=west,legend columns=-1}
    ]
    \addplot[smooth,color=black]
    table [x=freq, y=imag]{zeta_AHM-3v_L100_zoom.txt};
    \addlegendentry {AHM-3v};
    \addplot[dashed,smooth,color=black]
    table [x=freq, y=imag]{zeta_guess+ingard_L100_zoom.txt};
    \addlegendentry {Guess w. Ingard's end-corr.};
    \addplot[dashdotted,smooth,color=black]
    table [x=freq, y=imag]{zeta_guess+morse_L100_zoom.txt};
    \addlegendentry {Guess w. Morse' end-corr.};
    \draw[dashed, thin, color=black] (axis cs:350,0) -- (axis cs:380,0);
    \end{axis}
    \begin{axis}[
    xlabel=Frequency in Hz,
    xshift=7cm,
    xmin=1780, xmax=1820,
    ymin=-3, ymax=3,
    ylabel=$\Im(\zeta)$,
    line width=1pt,
    width=7.1cm, height=8cm,
    y label style={at={(axis description cs:-0.1,.5)},anchor=south}
    ]
    \addplot[smooth,color=black]
    table [x=freq, y=imag]{zeta_AHM-3v_L100_2.txt};
    \addplot[dashed,smooth,color=black]
    table [x=freq, y=imag]{zeta_guess+ingard_L100_2.txt};
    \addplot[dashdotted,smooth,color=black]
    table [x=freq, y=imag]{zeta_guess+morse_L100_2.txt};
    \draw[dashed, thin, color=black] (axis cs:1780,0) -- (axis cs:2000,0);
    \end{axis}
    \end{tikzpicture}
    \caption{Reactance $\Im(\zeta)$ as 
        as function of frequency in a neighborhood of the first (\textit{left}) and second (\textit{right}) characteristic frequency.
    }
    \label{fig:plot_impedance_zoom}
\end{figure}

In Figure~\ref{fig:plot_impedance} and in Figure~\ref{fig:plot_impedance_zoom}
we plot the reactance $\Im(\zeta)$ for a resonator depth $L=100\,\text{mm}$ and
the resistance $\Re(\zeta)$ (that is independent of $L$) for our model in
comparison with the formula of Guess with end corrections by Morse and Ingard
as functions of the frequency.  The reactance curve of our model is close to
the one of Guess for both end-point corrections, and closest to the end-point
correction of Ingard. However, the resistance curves differ from Guess'
formula.  The difference in resistance is the higher the lower is the
frequency, for 367\,Hz it is 19.1\% and for 1800\,Hz only 1.27\%.  We observe
that the resistance $\Re(\zeta_{\textsc{ckr-v}})$ tends for $\omega \to 0$ to a
positive constant, that is $\Re(\zeta_{\textsc{AHM-3v}})=0.1947$ in our study,
where $\zeta_{\textsc{Gue}}$ of Guess' model tends to zero.  Though, the
low-frequency behaviour of Guess' model seems to be unphysical, which is most
likely based on the wrong assumption that the quantity $\sqrt{\nu/\omega}$ is
uniformly small in $\omega$.




In the view of the works of Panton and Miller\cite{Panton.Miller:1975} the
resonance frequencies are the roots of $\Im(\zeta)$.  Guess' model predicts the
first resonance frequency with $5\,\text{Hz}$ lower using Morse' end-correction
and $4\,\text{Hz}$ higher using Ingard's end-correction than AHM-3v. The
second resonance frequency predicted by Guess' model is $2\,\text{Hz}$ lower
using Morse' end-correction and $5\,\text{Hz}$ higher using Ingard's
end-correction than our model.



\begin{table}[tb]
    \centering
    \begin{tabular}{cccccc}
        & \multicolumn{2}{c}{resonance frequencies} && \multicolumn{2}{c}{dissipation maxima} \\
        \cline{2-6}
        & \parbox[t]{4.5em}{\centering first}  & \parbox[t]{5em}{\centering second} &&
        \parbox[t]{4em}{\centering first} & \parbox[t]{4em}{\centering second} \\
        \hline
        \parbox[t]{5em}{\centering AHM-3v} & 367\,Hz & 
        1799
        \,Hz && 359\,Hz & 1793\,Hz\\
        \hline
        Guess w. Morse' end-corr. & 362\,Hz & 1797\,Hz && 351\,Hz & 1791\,Hz \\
        \hline
        Guess w. Ingard's end-corr. & 371\,Hz
        & 1804
        \,Hz && 360\,Hz & 1797\,Hz \\
        \hline
        \hline
    \end{tabular}
    \caption{Resonance frequencies that corresponds to zeros of $\Im(\zeta)$ and dissipation maxima for liner DC006$^\star$
        for a resonator chamber depth
        $L=100\,\text{mm}$.} 
    \label{table:computation_root_im_L100mm}
\end{table}

\subsection{Dissipation by an array of Helmholtz resonators in a duct}
\label{sec:study-dissipation}

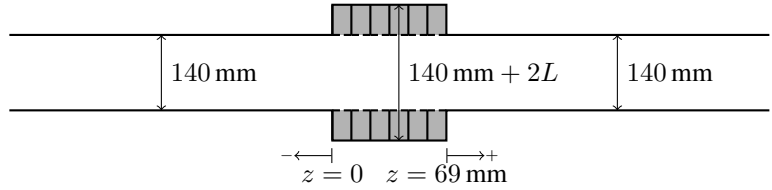
\begin{figure}[!bt]
    \centering
    \begin{tikzpicture}
    \fill[black!30!white] (-0.75,-0.5) rectangle (0.75,-0.9);
    \fill[black!30!white] (-0.75,0.5) rectangle (0.75,0.9);
    \draw[thick] (-5,-0.5) -- (-0.75,-0.5) -- (-0.75,-0.9) -- (0.75,-0.9) -- (0.75,-0.5)
    -- (5,-0.5);
    \draw[thick] (-5,0.5) -- (-0.75,0.5) -- (-0.75,0.9) -- (0.75,0.9) -- (0.75,0.5) --
    (5,0.5);
    \draw[<->] (3,-0.5) -- (3,0.5) node[pos=0.5,right] {$140\,\text{mm}$};
    \draw[<->] (-3,-0.5) -- (-3,0.5) node[pos=0.5,right] {$140\,\text{mm}$};
    \draw[<->] (0.125,-0.9) -- (0.125,0.9) node[pos=0.5,right] {$140\,\text{mm}+2L$};
    \draw[|->] (-0.75,-1.1) -- (-1.25,-1.1) node[pos=1.2] {{\tiny $-$}}
    node[pos=0,below] {$z=0$};
    \draw[|->] (0.75,-1.1) -- (1.25,-1.1) node[pos=1.2] {{\tiny $+$}}
    node[pos=0,below] {$z=69\,\text{mm}$};
    
    \foreach \i in {0,...,5}
    {
        \pgfmathparse{0.25*\i-0.75}\let\repos\pgfmathresult;
        \draw[thick] (\repos,-0.5) -- (\repos,-0.9);
        \draw[thick] (\repos,0.5) -- (\repos,0.9);
        \draw[thick] (\repos,-0.5) -- (\repos+0.1,-0.5);
        \draw[thick] (\repos+0.15,-0.5) -- (\repos+0.25,-0.5);
        \draw[thick] (\repos,0.5) -- (\repos+0.1,0.5);
        \draw[thick] (\repos+0.15,0.5) -- (\repos+0.25,0.5);
    }
    
    \end{tikzpicture}
    \caption{Considered domain for the duct acoustics. The grey region
        corresponds to the array of Helmholtz resonators.}
    \label{fig:figure_duct}
\end{figure}

Now, we simulate the transmission of a liner in an acoustic test duct (see\cite{Rienstra:2015} for fundamentals of duct acoustics) 
with circular cross-section  with a radius $R_d = 70\,\text{mm}$ (see Fig.~\ref{fig:figure_duct}).
The geometrical setting is similar to previous studies\cite{Lahiri:2014,Schmidt.Semin.ThoensZueva.Bake:2018}
on the duct DC006, where we replace the side chamber 
by an array of Helmholtz resonators. We call therefore the setting DC006$^\star$. 
This array of Helmholtz resonators has a total length $Z = 69\,\text{mm}$ and the other geometrical parameters 
were described in the beginning of this section.
We model the array by the impedance boundary conditions in~\eqref{eq:limit_term_complete_pressure} and compare it to the boundary conditions according to Guess' impedance formula.

We use the formulation for the acoustic pressure, see
(\ref{eq:limit_term_complete_pressure}), with a source term corresponding to an
incoming field $p_{\inc}(r,\theta,z) = \exp(\imath \omega z / c)$ from the
left. The scattered field is computed numerically using the mode matching
procedure as described in a previous work\cite{Semin.Thoens-Zueva.Schmidt:2017}
with $N=5$ radial modes.  We computed the eigenmodes numerically using the C++
Finite Elements Library
\textsc{Concepts}\cite{Frauenfelder.Lage:2002,conceptsweb}.

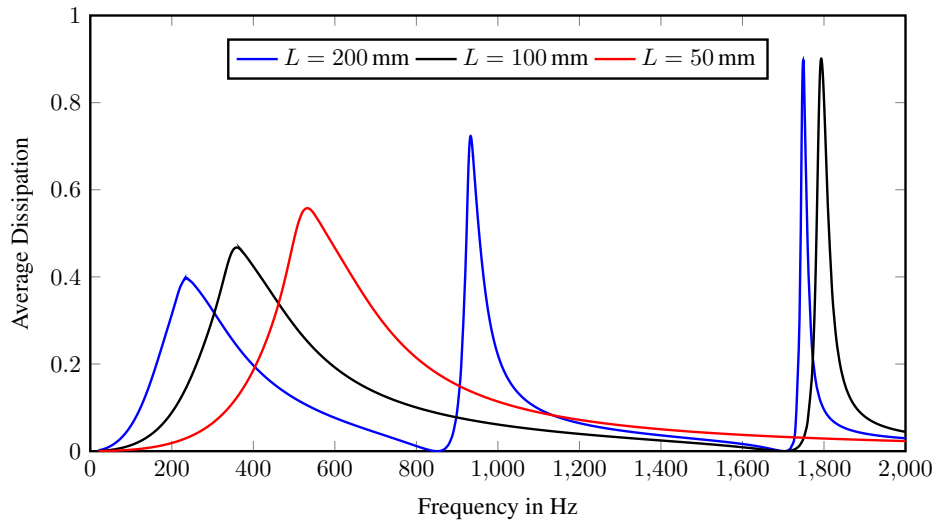
\begin{figure}[!bt]
    \centering
    \begin{tikzpicture}[scale=0.9]
    \begin{axis}[
    xlabel=Frequency in Hz,
    xmin=0,xmax=2000,
    ylabel=Average Dissipation,
    ymin=0,ymax=1,
    line width=1pt,
    width=13.5cm,height=8cm,
    legend style={at={(0.5,0.9)},anchor=center,legend columns=-1}]
    \addplot[smooth,color=blue]
    table [x=freq, y=D]{out_AHM-3v_L200mm.txt};
    \addlegendentry {$L=200\,\text{mm}$};
    \addplot[smooth,color=black]
    table [x=freq, y=D]{out_AHM-3v_L100mm.txt};
    \addlegendentry {$L=100\,\text{mm}$};
    \addplot[smooth,color=red]
    table [x=freq, y=D]{out_AHM-3v_L50mm.txt};
    \addlegendentry {$L=50\,\text{mm}$};

    \end{axis}
    \end{tikzpicture}
    \caption{Numerically computed average dissipation for the liner DC006$^\star$ 
        using the impedance model AHM-3v for different resonator depths $L$.}
    \label{fig:numerical_D_DC006_d1mm}
\end{figure}

We are interested in the energy dissipation of the liner that is defined as
\begin{equation*}
D := 1 - T - R,
\end{equation*}
where $T$ is the total transmitted energy and $R$ is the total reflected energy.
The energy dissipation is a global measure that depends on the damping properties of the array of Helmholtz resonators
as well as on wave profile and macroscopic geometric settings.

We plot in Fig.~\ref{fig:numerical_D_DC006_d1mm} the energy dissipation for
resonator chamber depths $L = 50\,\text{mm}, 100\,\text{mm}, 200\,\text{mm}$ as
functions of frequency for the proposed method AHM-3v.  The observe the first
maximum of energy dissipation at $240\,\text{Hz}$ for $L = 200\,\text{mm}$, %
at $359\,\text{Hz}$ for $L = 100\,\text{mm}$ and %
at $530\,\text{Hz}$ for $L = 50\,\text{mm}$, so at lower frequency for deeper
resonator chambers (while keeping the cross-section).  These frequencies are
slighly below the first resonance frequencies, which are at $235\,\text{Hz}$
for $L = 200\,\text{mm}$, %
at $367\,\text{Hz}$ for $L = 100\,\text{mm}$ and %
at $533\,\text{Hz}$ for $L = 50\,\text{mm}$.  The amplitude at the first energy
dissipation maximum is the higher the smaller the resonator chamber depth $L$.
Independent of the latter the energy dissipation goes to $0$ when decreasing
the frequency to $0$.

There is a first minimum of energy dissipation at $851\,\text{Hz}$ for
$L = 200\,\text{mm}$ and at $1702\,\text{Hz}$ for $L = 100\,\text{mm}$ where
the resonator chamber depth coincides with $\lambda/2$, where again $\lambda$
is the wave-length.  Here, the energy dissipation is $0$ as the impedance
boundary conditions becomes acoustic hard-wall conditions, \ie Dirichlet
boundary conditions for the normal component of the velocity, and the incident
field is passing the liner entirely transmitted. These frequencies correspond
to poles of $\Im(\zeta_{\textsc{ckr-v}})$ and are equal with those predicted by
the model of Guess' regardless of the choice of end-correction.

The second maxima of energy dissipation follows shortly the first minimum. It is observed %
at $930\,\text{Hz}$ for $L = 200\,\text{mm}$ and %
at $1793\,\text{Hz}$ for $L = 100\,\text{mm}$, and, hence, again shortly below 
the second resonance frequencies which are 
at $933\,\text{Hz}$ for $L = 200\,\text{mm}$ and %
at $1799\,\text{Hz}$ for $L = 100\,\text{mm}$. %
Another energy dissipation minimum is observed for $L=200\,\text{mm}$ at
$1702\,\text{Hz}$ followed by a maximum at $1750\,\text{Hz}$, which is again
shortly below the third resonance frequency $1749\,\text{Hz}$.

\begin{figure}[!bt]
    \begin{tikzpicture}[scale=0.9]
    \begin{axis}[
    xlabel=Frequency in Hz,
    xmin=250,xmax=450,
    ylabel=Average Dissipation,
    ymin=0,ymax=1,
    line width=1pt,
    width=6cm, height=8cm,
    y label style={at={(axis description cs:-0.15,.5)},anchor=south},
    legend style={at={(0.1,1.1)},anchor=west,legend columns=-1}]
    \addplot[smooth,color=black]
    table [x=freq, y=D]{out_AHM-3v_L100mm.txt};
    \addlegendentry {AHM-3v};
    \addplot[smooth,color=black,dashed]
    table [x=freq, y=D]{out_guess+ingard_L100mm.txt};
    \addlegendentry {Guess w. Ingard's end-corr.};
    \addplot[smooth,color=black,dashdotted]
    table [x=freq, y=D]{out_guess+morse_L100mm.txt};
    \addlegendentry {Guess w. Morse' end-corr.};
    
    \draw[dashed, thin, color=black] (axis cs:371,0) -- (axis cs:371,1); 
    
    
    \draw[thin, color=black] (axis cs:367,0) -- (axis cs:367,1); 
    \draw[dashdotted, thin, color=black] (axis cs:362,0) -- (axis cs:362,1); 
    
    \end{axis}
    \begin{axis}[
    xlabel=Frequency in Hz,
    xmin=1680,xmax=1920,
    ymin=0,ymax=1,
    xshift=5.8cm,
    line width=1pt,
    width=8.5cm,height=8cm,
    legend style={at={(1.1,0.8)},anchor=west}]
    \addplot[smooth,color=black]
    table [x=freq, y=D]{out_AHM-3v_L100mm.txt};
    \addplot[smooth,color=black,dashed]
    table [x=freq, y=D]{out_guess+ingard_L100mm.txt};
    \addplot[smooth,color=black,dashdotted]
    table [x=freq, y=D]{out_guess+morse_L100mm.txt};
    
    \draw[dashed, thin, color=black] (axis cs:1804,0) -- (axis cs:1804,1); 
    
    
    \draw[thin, color=black] (axis cs:1799.35,0) -- (axis cs:1799.35,1); 
    \draw[dashdotted, thin, color=black] (axis cs:1797,0) -- (axis cs:1797,1); 
    
    
    \end{axis}
    \end{tikzpicture}
    \caption{Numerically computed average dissipation for the liner configuration
        DC006$^\star$ for the model AHM-3v and the models from Rienstra and Guess
        close to the first (\textit{left}) and second (\textit{right}) resonance
        frequencies (see Table~\ref{table:computation_root_im_L100mm}) that are
        shown by vertical lines.  The diameter of the hole is
        $d_\delta = 1\,\text{mm}$ and the resonator depth is $L = 100\,\text{mm}$.}
    \label{fig:numerical_D_DC006_d1mm_mmodels}
\end{figure}
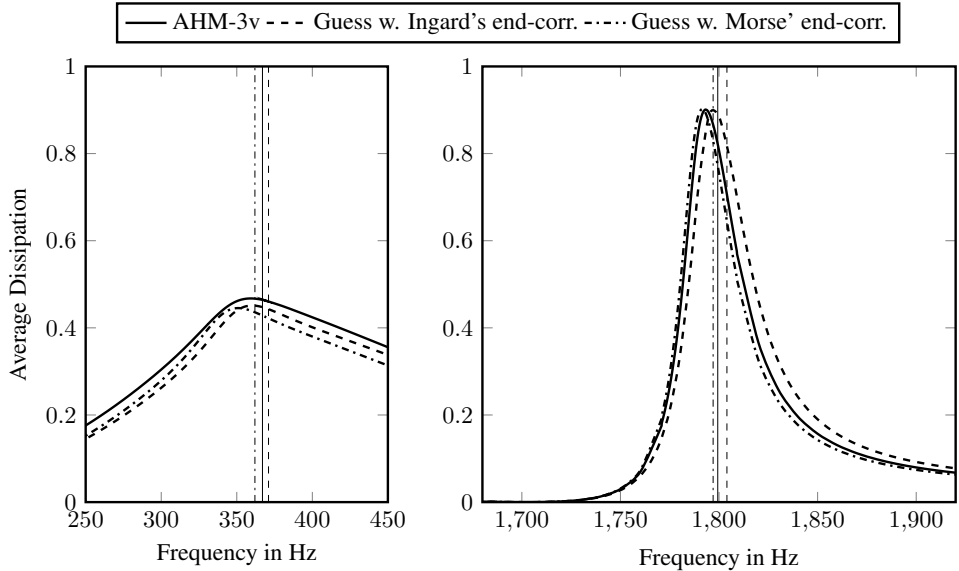

We plot in Fig.~\ref{fig:numerical_D_DC006_d1mm_mmodels} the energy dissipation
curves for resonator chamber length $L = 100\,\text{mm}$ in frequency windows
around the first and second maximum, this for the proposed method AHM-3v
together with those using the model of Guess with end-correctors of Morse and
Ingard. %
The vertical lines correspond to the respective resonance frequencies (see
Table~\ref{table:computation_root_im_L100mm}), \ie to the roots of the
reactance $\Im(\zeta)$ of the respective impedance.  The observed energy
dissipation maxima are very close for the different models (see
Table~\ref{table:computation_root_im_L100mm}), for Guess' model with Ingard's
correction it is only about $2,\text{Hz}$ lower for the first and about
$4\,\text{Hz}$ higher for the second maximum than the proposed model AHM-3v.
The predicted energy dissipation at the first maximum is around $3\,\%$ lower
for Guess' model with Ingards end corrector and around $5\,\%$
lower with Morse' end corrector than for the proposed model
AHM-3v.  This seems to be the consequence of the difference in resistance of the
formulas.  Around the second dissipation maxima the dissipation curves of the
different models are very close confirming that prediction of the first maximum
is most severe.

\section*{Conclusion}

We presented impedance boundary conditions that can be used to predict the damping properties as well as the resonance frequencies 
of periodic arrays of elongated Helmholtz resonators for low acoustic amplitudes in a stagnant gas. %
Considering a period $\delta$ that is small in comparison to the wave-length and even smaller diameter of the orifices, scaled as $\delta^2$,
and small viscosity that is scaled like $\delta^4$ we obtain a non-trivial limit for $\delta \to 0$, in difference to an homogenization with
only two geometric scales, \cf\cite{Bonnet.Drissi.Gmati:2005,Semin.Schmidt:2018}.
In this way the dominating effects are considered on each geometric scale, that are viscous effects and incompressible acoustic velocity around each hole,
only incompressibility in an intermediate zone above and below the perforated plate, one-dimension wave-propagation inside the resonance chamber and
pure acoustics wave-propagation away from the perforated plate. %
The separation of the scales give a choice of the viscous region, in difference to the approach by Lidoine~\etal\cite{Lidoine.Terasse.Abboud.Bennani:2007}. %

It turns out that the impedance boundary conditions depend mainly on effective Rayleigh conductity of the multiperforated plate\cite{Schmidt.Semin.ThoensZueva.Bake:2018}
and the reactance of the resonance chambers. %
The effective Rayleigh conductivity can be approximately computed by discretizating of a canonical problem in two half spaces separated by an wall of finite thickness 
except for a single hole, where the infinite domain has to be truncated. %

The derivation of the impedance boundary conditions is based on the two-scale convergence\cite{Nguetseng:1989,Allaire:1992}, which is to our 
knowledge the first time applied to a periodic transmission problem. The application of the two-scale convergence makes the proof of convergence 
less technical in comparison to the method of matched asymptotic expansions or the method of multiscale analysis, where viscous boundary layers
and the singular behaviour close to all edges of the geometry have to be considered. However, the justification is based on a stability assumption
for the $\delta$-dependent problem that shall be proved in a forthcoming article. %

In numerical simulations we have compared the derived impedance and computed dissipation with an established model~\cite{Guess:1975} of the acoustics community. %
The approach allows to integrate further effects as nonlinear convection for higher sound amplitudes or a gracing or bias flow.

\begingroup
\let\backupsection\section
\renewcommand{\section}[2]{\backupsection{#2}}%

\endgroup

\end{document}